\documentclass[]{amsart}
\pdfoutput=1

\usepackage{latexsym}                    
\usepackage{amssymb,accents}             
\usepackage[usenames,dvipsnames]{xcolor} 
\usepackage[colorlinks, 
            linkcolor=MidnightBlue, 
            citecolor=MidnightBlue, 
            urlcolor =BlueViolet,
           ]{hyperref}
\usepackage{booktabs}
\usepackage[all]{xy}
\usepackage{colonequals}

\newtheorem{theorem}{Theorem}[section]
\newtheorem{corollary}[theorem]{Corollary}
\newtheorem{lemma}[theorem]{Lemma}
\newtheorem{proposition}[theorem]{Proposition}

\theoremstyle{remark}
\newtheorem{definition}[theorem]{Definition}
\newtheorem{remark}[theorem]{Remark}

\theoremstyle{definition}
\newtheorem{algorithm}[theorem]{Algorithm}

\newlength{\algindent}
\newlength{\alglabel}
\newcounter{stepcount}

\newenvironment{alglist}
{\quad\begin{list}{\arabic{stepcount}.}%
{\leftmargin=\algindent\labelwidth=\algindent\itemsep=\smallskipamount\usecounter{stepcount}}}
{\end{list}}

\newcommand{\algin}{\item[\emph{Input}:]}
\newcommand{\algout}{\item[\emph{Output}:]}

\newcounter{substepcount}
\newenvironment{algsublist}
{\quad\begin{list}{(\/\rlap{\alph{substepcount}}\phantom{d}\/)}{\usecounter{substepcount}}}
{\end{list}}

\DeclareMathOperator{\Aut}{Aut}
\DeclareMathOperator{\Cross}{Cross}
\DeclareMathOperator{\Gal}{Gal}
\DeclareMathOperator{\PGL}{PGL}

\DeclareMathOperator{\Sym}{Sym}
\DeclareMathOperator{\Zeros}{Zeros}

\newcommand{\BF}{{\mathbf F}}
\newcommand{\BP}{{\mathbf P}}

\newcommand{\alphabar}{\bar{\alpha}}
\newcommand{\betabar}{\bar{\beta}}
\newcommand{\rhobar}{\bar{\rho}}
\newcommand{\sigmabar}{\bar{\sigma}}
\newcommand{\kbar}{\bar{k}}
\newcommand{\kstar}{k^\times}
\newcommand{\BFbar}{\overline{\BF}}
\newcommand{\Rhom}{R_\textup{hom}}

\newcommand{\Otilde}{\lowerwidetilde{O}}   

\newcommand{\projthree}[3]{{[#1\,{:}\,#2\,{:}\,#3]}}
\newcommand{\projtwo}[2]{{[#1\,{:}\,#2]}}
\newcommand{\twobytwo}[4]{
\bigl[\begin{smallmatrix}#1&#2\\#3&#4\end{smallmatrix}\bigr]}

\newcommand{\pz}{\phantom{0}}               
\newcommand{\zz}{\phantom{.00}}             
\newcommand{\z}{\phantom{.0}}             
\newcommand{\sz}{\rlap{$^*$}\phantom{.0}} 
\newcommand{\rs}{\rlap{$^*$}}

\newcommand{\kmodsquares}{k^\times/k^{\times2}}

\newcommand{\mybar}[1]{
  \mathchoice
  {#1\llap{$\overline{\phantom{\displaystyle\rm#1}}$}}
  {#1\llap{$\overline{\phantom{\textstyle\rm#1}}$}}
  {#1\llap{$\overline{\phantom{\scriptstyle\rm#1}}$}}
  {#1\llap{$\overline{\phantom{\scriptscriptstyle\rm#1}}$}}
}  
\renewcommand{\bar}{\mybar}

\DeclareMathSymbol{\widetildesym}{\mathord}{largesymbols}{"65}
\newcommand\lowerwidetildesym{%
  \text{\smash{\raisebox{-1.4ex}{%
    $\widetildesym$}}}}
        \newcommand\lowerwidetilde[1]{%
          \mathchoice
            {\accentset{\displaystyle\lowerwidetildesym}{#1}}
            {\accentset{\textstyle\lowerwidetildesym}{#1}}
            {\accentset{\scriptstyle\lowerwidetildesym}{#1}}
            {\accentset{\scriptscriptstyle\lowerwidetildesym}{#1}}
}

\begin{document}

\title[Enumerating hyperelliptic curves]
      {Enumerating hyperelliptic curves\\ over finite fields in quasilinear time}
      
\author{Everett W. Howe}
\address{Independent mathematician, 
         San Diego, CA 92104, USA}
\email{\href{mailto:however@alumni.caltech.edu}{however@alumni.caltech.edu}}
\urladdr{\href{http://ewhowe.com}{http://ewhowe.com}}

\date{20 June 2024}



\begin{abstract}
We present an algorithm that, for every fixed genus $g$, will enumerate all 
hyperelliptic curves of genus $g$ over a finite field $k$ of odd characteristic 
in quasilinear time; that is, the time required for the algorithm is 
$\Otilde(q^{2g-1})$, where $q=\#k$. Such an algorithm already exists in the
case $g=2$, thanks to work of Mestre and Cardona and Quer, and in the case
$g=3$, thanks to work of Lercier and Ritzenthaler. Experimentally, it appears
that our new algorithm is about two orders of magnitude faster in practice than
ones based on their work.
\end{abstract}

\maketitle


\section{Introduction}

There are many circumstances in which one may want to enumerate all 
hyperelliptic curves of a given genus over a given finite field.  One may wish
to determine whether a hyperelliptic curve with certain special properties 
exists --- for instance, with a certain number of
points~\cite[p.~393]{KodamaTopWashio2009}, or a certain zeta function, or a
certain $a$-number~\cite[\S4]{Frei2018}, or some other property of interest --- 
and explicit enumeration allows for a direct search. Or perhaps one may wish to
gather data about all hyperelliptic curves of a given genus over a given finite
field --- for example, in order to compute the distribution of the number of 
points on such curves, as in~\cite{BergstromHoweEtAl2024}, or to determine
experimental results~\cite{MaisnerNart2002} that can inspire future
theorems~\cite{HoweNartRitzenthaler2009}.

In this paper we present, for every fixed genus $g>1$, an algorithm to calculate
a list of all hyperelliptic curves of genus $g$ over a given finite field of odd
characteristic.\footnote{Hyperelliptic curves in characteristic~$2$ are
  different from those in other characteristics in some basic ways, and, as we
  discuss later, there already exists an efficient algorithm for enumerating
  them.}
  
\begin{theorem}
\label{T:main}
Fix an integer $g >1.$ Together, the algorithms we present in 
Section~\textup{\ref{S:enumerating}} provide a method for computing a complete 
list of hyperelliptic curves of genus~$g$ over $\BF_q$ for odd prime powers~$q$,
with each curve appearing exactly once up to isomorphism. The algorithms take
time $\Otilde(q^{2g-1})$ and space $O(q^{2g-1})$.
\end{theorem}

From~\cite[Proposition~7.1]{BrockGranville2001} and from the fact that a generic
hyperelliptic curve has automorphism group of order~$2$, we see that there are
roughly $2q^{2g-1}$ hyperelliptic curves of genus $g$ over~$\BF_q$, so our
algorithm runs in quasilinear time. 

We also present an apparently new explicit enumeration of all monic irreducible 
homogeneous bivariate quartics over a finite field $k$ of odd characteristic, up
to the natural action of $\PGL_2(k)$ (see Section~\ref{S:Weierstrass}), which 
may be of independent interest.

\begin{theorem}
\label{T:quartics}
Given an odd prime power~$q$, let $\gamma$ be a nonzero element of $\BF_{q^4}$
whose multiplicative order is $2(q^2-1)$ and let $\rho$ be an element of 
$\BF_{q^2}\setminus\BF_q$ with $\rho^2\in\BF_q$. Let $S_4$ be the union of the
two sets
\[
\bigl\{      (\gamma^i - 1)/(\gamma^i + 1) \ \vert\ \text{$i$ odd},\  1\le i\le (q+1)/2\bigr\}
\]
and
\[
\bigl\{ \rho (\gamma^i - 1)/(\gamma^i + 1) \ \vert\ \text{$i$ odd},\  1\le i\le (q-1)/2\bigr\}\,.
\]
Then the homogenized minimal polynomials of the elements of $S_4$ provide a
complete set of unique representatives for the action of $\PGL_2(\BF_q)$ on the
monic irreducible homogeneous bivariate quartics over $\BF_q$ .
\end{theorem}

We prove this in Section~\ref{S:quartics}. Also, here and elsewhere, when we say
we have a ``complete set of unique representatives'' for an action of a group on
a set, we mean that we have a set of orbit representatives that contains exactly
one member from each orbit.

Briefly, there are two main ideas behind the algorithms that give us
Theorem~\ref{T:main}. The first is that there is an easy way to tell whether two
irreducible homogeneous polynomials in $\BF_q[x,y]$ lie in the same orbit under 
the action of $\PGL_2(\BF_q)$ --- see Theorem~\ref{T:invariant}. We use this 
fact to reduce the problem of getting a list of all hyperelliptic curves
\emph{without} duplicates in quasilinear time to the problem of getting a list
of hyperelliptic curves with a \emph{bounded number} of duplicates in
quasilinear time. The second is that if one understands the cosets of 
$\PGL_2(\BF_q)$ in $\PGL_2(\BF_{q^2})$, and if one has a list of orbit 
representatives for $\PGL_2(\BF_{q^2})$ acting on irreducible homogenous
polynomials of degree~$n$ in $\BF_{q^2}[x,y]$, then one can get a list of orbit
representatives for $\PGL_2(\BF_q)$ acting on irreducible homogenous polynomials
of degree~$2n$ in $\BF_q[x,y]$ --- see Section~\ref{SS:n}. The point of this
observation is that the number of $\PGL_2(\BF_{q^2})$ orbits of irreducible
degree-$n$ polynomials over $\BF_{q^2}$ is on the order of $q^{2n-6}$, while the
number of $\PGL_2(\BF_q)$ orbits of degree-$2n$ polynomials over $\BF_q$ is
roughly $q^{2n-3}$, so the former is easier to compute than the latter.

The algorithm for enumerating hyperelliptic curves that we present is designed 
for simplicity of argument, rather than for efficiency of computation. In 
Section~\ref{S:speedups} we describe modifications that will make the algorithm
more efficient, and in Section~\ref{S:genus2} we present even more details and 
timings for the case of genus~$2$ and genus~$3$.

An obvious issue with our algorithm, as we present it here, is that it requires
$O(q^{2g-1})$ space. This is because our initial computations often produce some
duplicate entries, and we eliminate these duplicates by collecting all the
output, computing some invariants, and then discarding entries whose invariants 
have already been seen. In a followup paper~\cite{Howe2024}, we show how it is 
possible to modify the techniques presented here in order to get a quasilinear
time algorithm for computing genus-$g$ hyperelliptic curves over $\BF_q$ that
only requires $O(\log q)$ space. Our implementation of the genus-$2$ case of the 
algorithm from the present paper uses a basic version of the ideas
from~\cite{Howe2024}, and it would not be hard to modify the genus-$2$ code we
provide in~\cite{HoweRepo2024} so that it requires only $O(q)$ space.

We start by considering various reductions, special cases, and lemmas.
In Section~\ref{S:Weierstrass} we show that enumerating hyperelliptic curves of 
genus $g$ over a finite field $k$ of odd characteristic can be reduced to 
enumerating Galois-stable sets of $2g+2$ elements of~$\BP^1(\kbar)$ up to the
natural action of $\PGL_2(k)$. The rest of the paper is therefore concerned
mostly with the latter problem, which we solve for all finite fields, not just 
those of odd characteristic.
In Section~\ref{S:quartics} we prove Theorem~\ref{T:quartics}, as well as
similar results for quartics with one or two irreducible quadratic factors
and generalizations to characteristic~$2$.
In Section~\ref{S:invariant} we introduce an invariant for $\PGL_2(\BF_q)$
orbits of monic irreducible homogeneous bivariate polynomials of degree~$n\ge 4$
over~$\BF_q$, and we use it to provide a very straightforward algorithm for
giving a complete set of unique representatives for these orbits in 
time~$\Otilde(q^{n-2})$. 
In Sections~\ref{S:CosetRepsPGL2} and~\ref{S:CosetRepsPGLp} we give an explicit
enumeration of a complete set of unique representatives for the cosets of 
$\PGL_2(\BF_q)$ in $\PGL_2(\BF_{q^p})$ for primes~$p$. The case $p=2$ is the key
result needed for the most difficult case of our algorithm to enumerate
hyperelliptic curves.
In Section~\ref{S:enumerating} we use the results of the earlier sections to 
present a collection of algorithms that, together, give a complete set of unique
representatives for the $\PGL_2(\BF_q)$ orbits of Galois stable sets of ${2g+2}$
elements of $\BP^1(\BFbar_q)$.
We close with Sections~\ref{S:speedups} and~\ref{S:genus2}, described above.

For many of our algorithms we require an easily computable total ordering of
the elements of $\BF_q$ or of polynomials in $\BF_q[x]$ or in $\BF_q[x,y]$. We
will always denote such an ordering by ``$<$'' and we leave the reader to choose
their favorite one. Also, if the proof of a proposition is clear, we indicate
that the proof will be skipped by including an end-of-proof mark in the 
statement of the result.

\subsubsection*{Acknowledgements}
I am grateful to the referees for ANTS, who provided helpful feedback that 
improved this paper.

\section{Hyperelliptic curves and Weierstrass points}
\label{S:Weierstrass}

We begin by setting some general notation. Given a finite field~$k$, let $R$ be
the graded polynomial ring $k[x,y]$, with the grading given by the degree. For
each $n$ let $R_{n}$ be the set of homogeneous polynomials in $R$ of degree $n$
and let $\Rhom$ be the union of the $R_{n}$. We say that $f\in \Rhom$ is
\emph{monic} if $f(x,1)$ is monic as a univariate polynomial, and we say that
$f$ is \emph{separable} if in $\kbar[x,y]$ it can be written as a constant times
a product of distinct monic linear factors. We say that $\alpha\in\kbar$ is a 
\emph{root} of $f$ if $f(\alpha,1) = 0$, and that 
$\projtwo{\alpha}{\beta}\in\BP^1(\kbar)$ is a \emph{zero} of $f$ if 
$f(\alpha,\beta) = 0$.

We also define a left action of $\PGL_2(k)$ on $\Rhom/\kstar$ as follows. If 
$\Gamma$ is an element of $\PGL_2(k)$ represented by a matrix 
$M\colonequals \twobytwo{a}{b}{c}{d}$ and if $f$ lies in $\Rhom$, we define
$\Gamma(f\bmod\kstar)$ to be the class in $\Rhom/\kstar$ of $f(dx-by,-cx+ay)$. 

Every class of $\Rhom/\kstar$ contains a unique monic element, and we define 
an action of $\PGL_2(k)$ on the monic elements of $\Rhom$ by writing 
$\Gamma(f) = g$ when $g$ is the monic element of $\Gamma(f \bmod \kstar)$. If
$\Gamma$ is represented by a matrix $M\colonequals\twobytwo{a}{b}{c}{d}$ and if
$\Gamma(f) = g$, then $f(dx-by,-cx+ay) = e g(x,y)$ for some $e\in \kstar$, and 
if we choose a different matrix to represent $\Gamma$, then the constant $e$
will be multiplied by an $n$th power. When $n$ is even, the class of $e$ in 
$\kmodsquares$ therefore depends only on $\Gamma$, and we denote this square 
class by~$s_{\Gamma,f}$.

Given a separable polynomial $f\in \Rhom$, we let $\Zeros(f)$ denote the set of
zeros of $f$ in $\BP^1(\kbar)$, so that $\Zeros(f)$ consists of the roots of $f$
under the usual inclusion $\kbar\subset\BP^1(\kbar)$ given by 
$\alpha\mapsto \projtwo{\alpha}{1}$, together with 
$\infty\colonequals\projtwo{1}{0}\in\BP^1(\kbar)$ if $f$ is divisible by~$y$.
For every integer $n\ge 0$ we let $\Sym^n(k)$ denote the set of all 
Galois-stable sets of $n$ distinct points in $\BP^1(\kbar)$, so that the natural
action of $\PGL_2(k)$ on $\BP^1(\kbar)$ leads to an action of $\PGL_2(k)$ on 
$\Sym^n(k)$. We see that $\Zeros$ gives us a bijection between the set of monic 
separable polynomials in $R_{n}$ and the set $\Sym^n(k)$, and we check that
$\Zeros(\Gamma(f)) = \Gamma(\Zeros(f))$ for all $f\in\Rhom$.

Our goal in this paper is to produce an algorithm to enumerate hyperelliptic
curves of a given genus~$g$ over finite fields $k$ of odd characteristic. As a
first step, we reduce this problem to the problem of computing representatives
for all $\PGL_2(k)$ orbits of $\Sym^n(k)$, where $n=2g+2$.

Let $k$ be a finite field of odd characteristic and let $C$ be a hyperelliptic
curve over~$k$, that is, a curve of genus~$g>1$ with a degree-$2$ map to a curve
of genus~$0$. Every genus-$0$ curve over a finite field is isomorphic 
to~$\BP^1$, and since our $k$ has odd characteristic $C$ can be written in the
form $z^2 = \tilde{f}$, where $\tilde{f}\in k[x]$ is a separable polynomial of
degree $2g+1$ or $2g+2$. We can rewrite this as a model $z^2 = f(x,y)$ in
weighted projective space by homogenizing $\tilde{f}$ into a polynomial
$f\in R_{2g+2}$; here we give the coordinates $x$ and $y$ weight $1$ and the 
coordinate $z$ weight $g+1$. Then the map 
$\projthree{x}{y}{z}\mapsto\projtwo{x}{y}$ gives the double cover $C\to\BP^1$,
and the points of $\BP^1$ that ramify in this map are exactly the elements of
$\Zeros(f)$. If $C_1$ and $C_2$ are two hyperelliptic curves, given by equations
$z^2 = f_1$ and $z^2 = f_2$, then every isomorphism from $C_1$ to $C_2$ can be
written in the form
\[
\projthree{x}{y}{z} \mapsto \projthree{ax+by}{{cx+dy}}{ez}\,,
\]
where $ad-bc$ and $e$ are nonzero, and where 
\begin{equation}
\label{EQ:HEisomorphism}
e^2 f_1(x,y) = f_2(ax+by,cx+dy)\,;
\end{equation}
see~\cite[Corollary~7.4.33]{Liu2002}. Then 
$e^2f_1(dx-by,-cx+ay) = (ad-bc)^{2g+2}f_2(x,y)$, so we see that
$f_2\bmod \kstar = \Gamma(f_1\bmod \kstar)$, where $\Gamma\in\PGL_2(k)$ is the 
element represented by the matrix $M\colonequals\twobytwo{a}{b}{c}{d}$. Thus, if
$C_1$ and $C_2$ are isomorphic, the element $\Gamma$ of $\PGL_2(k)$ takes the 
ramification points of $C_1\to\BP^1$ to the ramification points of 
$C_2\to\BP^1$. Conversely, if $\twobytwo{a}{b}{c}{d}\in \PGL_2(k)$ takes the 
ramification points of $C_1\to\BP^1$ to the ramification points of
$C_2\to\BP^1$, then there is an $e\in\kbar$, with $e^2\in k$, that 
makes~\eqref{EQ:HEisomorphism} hold. If $e$ lies in~$k$, we have an isomorphism 
between $C_1$ and $C_2$; if $e$ does not lie in $k$, we have an isomorphism
between $C_1$ and the \emph{quadratic twist} of~$C_2$, that is, the curve 
$y^2 = \nu f_2$, where $\nu$ is a nonsquare in~$k$. (In general, a \emph{twist}
of a curve $C$ over a finite field $k$ is another curve over $k$ that becomes 
isomorphic to $C$ when the base field is extended to an algebraic closure 
of~$k$. Twists of $C$ correspond to elements of the cohomology set 
$H^1(\Gal(\kbar/k),\Aut_{\kbar} C)$ --- see~\cite[\S III.1.3]{Serre1994} --- and
by ``the quadratic twist'' of a hyperelliptic curve we mean the twist 
corresponding to the cocycle that sends the Frobenius element of $\Gal(\kbar/k)$
to the hyperelliptic involution. Note that sometimes the quadratic twist may in
fact be the trivial twist.)

Thus, we have a map from the set of isomorphism classes of genus-$g$ 
hyperelliptic curves over $k$ to the set of $\PGL_2(k)$ orbits of $\Sym^n(k)$,
where $n=2g+2$. This map is clearly surjective, and the $\PGL_2(k)$ orbit of an 
element $A$ of $\Sym^n(k)$ has at most two preimages in the set of isomorphism
classes of hyperelliptic curves: the isomorphism classes of $y^2 = f$ and of
$y^2=\nu f$, where $f$ is the unique monic polynomial with $\Zeros(f) = A$ and 
where $\nu$ is a nonsquare in~$k$. (We say ``at most two'' preimages because,
as we noted above, these two curves may be isomorphic to one another.)

Whether an element $A$ of $\Sym^n(k)$ has one or two preimages is easy to
determine: Let $f\in R_n$ be the unique monic polynomial with $\Zeros(f) = A$.
Compute all elements $\Gamma$ of $\PGL_2(k)$ that take the set $A$ to itself;
at worst this takes time ${O(n(n-1)(n-2))}$, and since $n$ is fixed in our
context, this is $O(1)$ operations. For each such $\Gamma$ compute the element
$s_{\Gamma,f}$ of~$\kmodsquares$. If any of these elements is nontrivial, then
the curve $y^2 = f$ is isomorphic to its twist $y^2 = \nu f$. (Compare 
to~\cite[Lemma~1.2]{Nart2009}.)

This shows that if we can compute a complete set of unique representatives for 
the orbits of $\PGL_2(\BF_q)$ acting on $\Sym^{2g+2}(\BF_q)$ in time 
$\Otilde(q^{2g-1})$, we can also compute a complete set of unique 
representatives for the hyperelliptic curves of genus $g$ over $\BF_q$ in time 
$\Otilde(q^{2g-1})$. Thus, for the rest of this paper we focus on enumerating
the $\PGL_2(k)$ orbits of $\Sym^n(k)$, for $n=2g+2$. In particular, for our 
application to enumerating hyperelliptic curves we only need to consider even 
$n$ that are at least~$6$. The algorithms that we present for the latter problem
work for all finite fields, not just those of odd characteristic.

\begin{remark}
Over a finite field $\BF_q$ of characteristic~$2$ there are still roughly
$2q^{2g-1}$ hyperelliptic curves of genus~$g$, but it is much easier to
enumerate them in quasilinear time than it is in odd characteristic, because the
ramification divisor of the hyperelliptic structure map $C\to\BP^1$ is supported
on at most $g+1$ points. Enumerating the possible ramification divisors up to
the action of $\PGL_2(\BF_q)$ in time $\Otilde(q^{2g-1})$ is therefore much 
simpler than in odd characteristic, and enumerating the curves with a given
ramification divisor is relatively straightforward. The algorithm of 
Xarles~\cite{Xarles2020} follows this outline; it has been implemented by him in
genus~$4$, by Dragutinovi\'{c}~\cite{Dragutinovic2024} in genus~$5$, and by 
Huang, Kedlaya, and Lau~\cite{HuangKedlayaLau2024} in genus~$6$.
\end{remark}

\section{Results for quartic polynomials}
\label{S:quartics}

In this section we prove Theorem~\ref{T:quartics}. We also prove similar results
that give complete sets of unique representatives for the action of 
$\PGL_2(\BF_q)$ on monic homogeneous quartics that have one or two irreducible
quadratic factors, and we state generalizations to finite fields of
characteristic~$2$. We begin with an elementary lemma.

\begin{lemma}
\label{L:involution}
Let $k$ be a field and let $a$, $b$, $c$, and $d$ be distinct elements 
of~$\BP^1(k)$. Then there is a unique element of $\PGL_2(k)$ that swaps $a$ with
$b$ and $c$ with~$d$, and this element is an involution.
\end{lemma}

\begin{proof}
An element of $\PGL_2(k)$ is determined by where it sends three distinct
elements of $\BP^1(k)$, so the uniqueness is automatic, and we need only prove
existence. By using the action of $\PGL_2(k)$, we see that it suffices to prove
the lemma in the case where $a=\infty$, $b=0$, and $c=1$. Then the element
$\twobytwo{0}{d}{1}{0}$ of $\PGL_2(k)$ is an involution that swaps $a$ with $b$
and $c$ with $d$.
\end{proof}

\begin{corollary}
\label{C:involutiongalois}
Let $q$ be a prime power and let $\alpha$, $\beta$, $\gamma$, and $\delta$ be
distinct elements of $\BFbar_q$ such that
$
  \{\{\alpha,  \beta  \},\{\gamma,  \delta  \}\} 
= \{\{\alpha^q,\beta^q\},\{\gamma^q,\delta^q\}\}\,.
$
Then there is a unique element of $\PGL_2(\BF_q)$ that swaps $\alpha$ with
$\beta$ and $\gamma$ with~$\delta$, and this element is an involution.
\end{corollary}

\begin{proof}
Let $\Gamma\in\PGL_2(\BFbar_q)$ be the involution that swaps $\alpha$ with
$\beta$ and $\gamma$ with~$\delta$. Then $\Gamma^{(q)}$, by which we mean the 
element of $\PGL_2(\BFbar_q)$ obtained by taking a representative matrix for
$\Gamma$ and replacing all of its entries by their $q$th powers, is also an
involution that swaps $\alpha$ with $\beta$ and $\gamma$ with~$\delta$, because
of the equality of sets in our hypothesis. By the uniqueness property in
Lemma~\ref{L:involution}, it follows that $\Gamma^{(q)} = \Gamma$ in 
$\PGL_2(\BFbar_q)$, from which we see that $\Gamma$ actually lies in 
$\PGL_2(\BF_q)$.
\end{proof}

\begin{proof}[Proof of Theorem~\textup{\ref{T:quartics}}]
First we show that every monic irreducible quartic in $\BF_q[x]$ can be
transformed into one of the quartics in the statement of the theorem; this is 
equivalent to showing that every irreducible quartic has a root in $\BF_{q^4}$ 
that can be moved by $\PGL_2(\BF_q)$ to an element of the set~$S_4$ from the 
theorem.

In the statement of the theorem we chose an element $\rho$ of 
$\BF_{q^2}\setminus\BF_q$ with $\rho^2\in \BF_q$. Let $\nu = \rho^2$, so that
$\nu$ is a nonsquare in~$\BF_q$.

Let $f$ be an irreducible quartic in $\BF_q[x]$, let $\alpha$ be a root of $f$
in $\BF_{q^4}$, and set $\alpha_1 \colonequals \alpha$, 
$\alpha_2 \colonequals \alpha_1^q$, $\alpha_3 \colonequals \alpha_2^q$, and
$\alpha_4 \colonequals \alpha_3^q$. By Corollary~\ref{C:involutiongalois} there 
is a unique involution $\Gamma$ in $\PGL_2(\BF_q)$ that swaps $\alpha_1$ with
$\alpha_3$ and $\alpha_2$ with~$\alpha_4$. We refer to this as the involution 
associated to~$f$.

Let $\Phi$ be an element of $\PGL_2(\BF_q)$. Then the involution associated to 
$\Phi(f)$ is $\Phi\Gamma\Phi^{-1}$. Since every involution in $\PGL_2(\BF_q)$ is
conjugate either to $\twobytwo{0}{1}{1}{0}$ or $\twobytwo{0}{\nu}{1}{0}$, we can
choose $\Phi$ so that $\Phi\Gamma\Phi^{-1}$ is one of these two standard
involutions. Now we replace $f$ with $\Phi(f)$, $\alpha$ with $\Phi(\alpha)$,
and $\Gamma$ with $\Phi\Gamma\Phi^{-1}$.

Suppose $\Gamma = \twobytwo{0}{1}{1}{0}$. The subgroup of $\PGL_2(\BF_q)$ that
stabilizes $\Gamma$ under conjugation is
\[
H_1 \colonequals  
    \Biggl\{ \begin{bmatrix}  a& b\\           b&           a\end{bmatrix} \ \bigg\vert \ a^2 \ne b^2 \Biggr\}
    \cup  
    \Biggl\{ \begin{bmatrix} -a&-b\\\phantom{-}b&\phantom{-}a\end{bmatrix} \ \bigg\vert \ a^2 \ne b^2 \Biggr\}\,.
\]
We would like to apply an element of $H_1$ to $\alpha$, and if necessary replace
$\alpha$ with one of its conjugates, to put $\alpha$ into a standard form. We 
accomplish this by considering the function 
$\BP^1(\BF_{q^4})\to\BP^1(\BF_{q^4})$ given by applying the element
$\Psi\colonequals \twobytwo{\phantom{-}1}{1}{-1}{1}$ of $\PGL_2(\BF_q)$. Since 
$\Psi$ conjugates $\Gamma$ to the involution 
$\twobytwo{-1}{0}{\phantom{-}0}{1}$, we see that 
${\Psi(\alpha)^{q^2} = -\Psi(\alpha)}$. This shows that 
$\Psi(\alpha)^{q^2-1} = -1$, so the multiplicative order of $\Psi(\alpha)$ is
even and divides $2(q^2-1)$. It follows that we may write 
$\Psi(\alpha) = \gamma^i$ for some odd integer $i$ with $0 < i < 2(q^2-1)$.

Let 
\[
H_1'\colonequals \Psi H_1\Psi^{-1}
    = 
    \Biggl\{ \begin{bmatrix} a&0\\0&1\end{bmatrix} \ \bigg\vert \ a\in\BF_q^\times \Biggr\}
    \cup  
    \Biggl\{ \begin{bmatrix} 0&a\\1&0\end{bmatrix} \ \bigg\vert \ a\in\BF_q^\times \Biggr\}\,.
\]
Then applying an element of $H_1$ to $\alpha$ corresponds to applying an element 
of $H_1'$ to $\Psi(\alpha)$, and the elements of $H_1'$ either multiply 
$\Psi(\alpha)$ by an element of $\BF_q^\times$ or replace $\Psi(\alpha)$ with 
its inverse times an element of $\BF_q^\times$. Since the elements of
$\BF_q^\times$ are the powers of $\gamma^{2(q+1)}$, these two actions show that 
there is a $\Phi'\in H_1'$ such that $\Phi'(\Psi(\alpha)) = \gamma^i$ for an odd 
integer $i$ with $0<i< q+1$. If we let $\Phi = \Psi^{-1}\Phi'\Psi\in H_1$ and
replace $\alpha$ with $\Phi(\alpha)$, we find that $\Psi(\alpha) = \gamma^i$ for
this~$i$.

Finally, replacing $\alpha$ with $\alpha^q$ has the effect of replacing $i$
with~$iq$. If we write $i = 2h+1$ we see that
\[
iq \equiv 2hq+q\equiv -2h + q \equiv -(2h+1) + (q+1)\equiv q+1-i\bmod 2(q+1)\,,
\]
so by replacing $\alpha$ with its conjugate $\alpha^q$, if necessary, and then
modifying the new $\alpha$ by an element of $H_1$, we find that we may assume 
that $0<i\le (q+1)/2$.

We see that every irreducible quartic whose associated involution is
$\twobytwo{0}{1}{1}{0}$ has a root in the $\PGL_2(\BF_q)$ orbit of 
$\Psi^{-1}(\gamma^{i}) = (\gamma^{i}-1)/(\gamma^{i}+1)$, for some odd $i$ with 
$0<i\le (q+1)/2$. Moreover, from our analysis it is clear that different values
of $i$ in this range produce quartics in distinct $\PGL_2(\BF_q)$ orbits.

Next, suppose the involution associated with a quartic $f$ is
$\Gamma = \twobytwo{0}{\nu}{1}{0}$. The subgroup of $\PGL_2(\BF_q)$ that
stabilizes $\Gamma$ under conjugation is
\[
H_\nu \colonequals  
   \Biggl\{ \begin{bmatrix}  a& b\nu\\           b&           a\end{bmatrix} \ \bigg\vert \ a^2 \ne \nu b^2 \Biggr\}
   \cup  
   \Biggl\{ \begin{bmatrix} -a&-b\nu\\\phantom{-}b&\phantom{-}a\end{bmatrix} \ \bigg\vert \ a^2 \ne \nu b^2 \Biggr\}\,.
\]
As before, we would like to apply an element of $H_\nu$ to $\alpha$, and if
necessary replace $\alpha$ with one of its conjugates, to put $\alpha$ into a
standard form. This time we consider the function 
$\BP^1(\BF_{q^4})\to\BP^1(\BF_{q^4})$ given by applying the element
$\Psi\colonequals \twobytwo{\phantom{-}1}{\rho}{-1}{\rho}$ of
$\PGL_2(\BF_{q^2})$, where $\rho$ is the element of $\BF_{q^2}$ chosen in the 
statement of the theorem and $\nu=\rho^2$. Then $\Psi$ conjugates $\Gamma$ to
the involution $\twobytwo{-1}{0}{\phantom{-}0}{1}$, and once again we have 
$\Psi(\alpha)^{q^2} = -\Psi(\alpha)$. As before, we see that 
$\Psi(\alpha)^{q^2-1} = -1$, so $\Psi(\alpha)$ has multiplicative order dividing
$2(q^2-1)$. Once again we may write $\Psi(\alpha) = \gamma^i$ for some odd
integer $i$ with $0< i < 2(q^2-1)$.

Let $H_\nu'\colonequals \Psi H_\nu\Psi^{-1}$ and let $N$ be the kernel of the
norm map from $\BF_{q^2}^\times$ to $\BF_{q}^\times$. We check that
\[
H_\nu' = 
   \Biggl\{ \begin{bmatrix} a&0\\0&1\end{bmatrix} \ \bigg\vert \ a\in N \Biggr\}
   \cup  
   \Biggl\{ \begin{bmatrix} 0&a\\1&0\end{bmatrix} \ \bigg\vert \ a\in N \Biggr\}\,.
\]
Then applying an element of $H_\nu$ to $\alpha$ corresponds to applying an
element of $H_\nu'$ to $\Psi(\alpha)$, and the elements of $H_\nu'$ either
multiply $\Psi(\alpha)$ by an element of $N$ or replace $\Psi(\alpha)$ with its
inverse times an element of $N$.

The elements of $N$ are the powers of $\gamma^{2(q-1)}$, so arguing as before we
find that we may replace $\alpha$ with $\Phi(\alpha)$ for some $\Phi\in H_\nu$
so that $\Psi(\alpha)= \gamma^i$ for an odd $i$ with $0<i<q-1$. We check that
$\Psi(\alpha^q) = 1/\Psi(\alpha)^q$, so replacing $\alpha$ with $\alpha^q$ has 
the effect of replacing $i$ with~$-iq$. If we write $i = 2h-1$ we see that
\[
-iq \equiv -2hq+q\equiv -2h + q \equiv (-2h+1) + (q-1)\equiv q-1-i \bmod 2(q-1)\,,
\]
so by replacing $\alpha$ with its conjugate $\alpha^q$, if necessary, we find
that we may assume that $0<i\le (q-1)/2$.

We see that every irreducible quartic whose associated involution is
$\twobytwo{0}{\nu}{1}{0}$ has a root in the $\PGL_2(\BF_q)$ orbit of 
$\Psi^{-1}(\gamma^{i}) = \rho (\gamma^{i}-1)/(\gamma^{i}+1)$, for some odd $i$
with $0<i\le (q-1)/2$, and different values of $i$ in this range give 
irreducible quartics that are not equivalent to one another under the action of 
$\PGL_2(\BF_q)$.
\end{proof}

\begin{remark}
\label{R:j}
For a separable quartic $f = x^4 + ax^3y + bx^2y^2 + cxy^3 + dy^4$, we let 
$j(f)$ denote the $j$-invariant of the Jacobian of the genus-$0$ curve 
$z^2 = f$. One can show that $j(f) = 256(b^2-3ac+12d)^3/\Delta$, where $\Delta$ 
is the discriminant of~$f$, and clearly $j(\Gamma(f)) = j(f)$ for all 
$\Gamma\in\PGL_2(\BF_q)$, because the curve $z^2 = \Gamma(f)$ is geometrically
isomorphic to $z^2 = f$. Using arguments from~\cite[\S3]{Howe2017}, one can show
that $j$ takes different values on irreducible quartics that are not in the
same $\PGL_2(\BF_q)$ orbit.
\end{remark}

For products of two distinct irreducible quadratics, we have a similar result.

\begin{theorem}
\label{T:twoquadlist}
Given an odd prime power $q$, let $\zeta$ be a generator of $\BF_{q^2}^\times$,
let $\rho$ be an element of $\BF_{q^2}\setminus\BF_q$ with $\rho^2\in\BF_q$,
and let $\nu=\rho^2$. Let 
\[
S_{22} = \{ \rho(\zeta^i-1)/(\zeta^i+1) \ \vert \ 0<i\le (q-1)/2 \}
\]
and let $T_{22}$ be the set of homogenized minimal polynomials of the elements
of $S_{22}$. Then the set $\{(x^2-\nu y^2)g \ \vert\ g \in T_{22}\}$ is a 
complete set of unique representatives for the action of $\PGL_2(\BF_q)$ on the 
homogeneous quartics that factor into a product of two distinct monic 
irreducible quadratics.
\end{theorem}

\begin{proof}
Let $f$ be a homogeneous quartic that factors into a product of two monic
irreducible quadratics, so that the roots of $f$ are $\alpha$, $\alphabar$,
$\beta$, and $\betabar$, for two elements $\alpha$ and $\beta$ in $\BF_{q^2}$
with conjugates $\alphabar$ and $\betabar$. We will show that there is a unique
element $\sigma$ in $S_{22}$ such that the set 
$\{\alpha,\alphabar,\beta,\betabar\}$ can be sent to 
$\{\rho,\rhobar,\sigma,\sigmabar\}$ by an element of $\PGL_2(\BF_q)$. This will
be enough to prove the theorem.

By replacing $f$ with its image under an element of $\PGL_2(\BF_q)$ we may
assume that $\alpha = \rho$. The subgroup of $\PGL_q(\BF_q)$ that fixes the set
$\{\rho,\rhobar\}$ is the group $H_\nu$ from the proof of 
Theorem~\ref{T:quartics}. Let 
$\Psi\colonequals \twobytwo{\phantom{-}1}{\rho}{-1}{\rho}$. Arguing as in the
proof of Theorem~\ref{T:quartics} we find that there is a unique element
$\Gamma$ of $H_\nu$ so that we can write $\Psi(\Gamma(\beta))$ as $\zeta^i$ for
an integer $i$ with $0<i\le (q-1)/2$. 

Since by Corollary~\ref{C:involutiongalois} there is an element of 
$\PGL_2(\BF_q)$ that swaps $\alpha$ and $\alphabar$ with $\beta$ and $\betabar$,
we would have gotten the same value of $i$ if we had normalized $\beta$ and
$\betabar$ to be $\rho$ and $\rhobar$ at the beginning of our argument, instead
of $\alpha$ and~$\alphabar$. Thus the value of $i$ we obtain truly depends only
on the $\PGL_2(\BF_q)$ orbit of~$f$.

This shows that we may assume that $\beta = \Psi^{-1}(\zeta^i)$ is an element 
of~$S_2$, and different elements of $S_2$ correspond to different
$\PGL_2(\BF_q)$ orbits. The theorem follows.
\end{proof}

\begin{remark}
\label{R:mu}
If $f$ is a homogeneous quartic in $\BF_q[x,y]$ that can be factored into the
product of two monic irreducible quadratics $x^2 + sxy + ty^2$ and
$x^2 + uxy + vy^2$, we define
\[
\mu(f)\colonequals \frac{(su - 2t - 2v)^2}{(s^2-4t)(u^2-4v)}\,,
\]
and we check that $j(f) = 64(\mu(f)+3)^3/(\mu(f)-1)^2$, where $j(f)$ is as in
Remark~\ref{R:j}. Note that $\mu(f)$ is a square, because the two factors in the
denominator are the discriminants of the irreducible factors of $f$.

We leave it to the reader to check that for $f$ of this form we have 
$\mu(\Gamma(f)) = \mu(f)$ for every $\Gamma\in \PGL_2(\BF_q)$, so $\mu$ is an 
invariant of the $\PGL_2(\BF_q)$ orbits of such quartics. Given any square $d$ 
in $\BF_q$ other than~$1$, we check that a quartic 
$f \colonequals (x^2-\nu y^2)(x^2 + uxy + vy^2)$ satisfies $\mu(f) = d$ if and
only if $(u,v)$ lies on a certain nonsingular conic. Nonsingular conics over 
finite fields have rational points not on the line at infinity, so there are 
values of $u$ and $v$ that give a quartic for which $\mu$ attains the value~$d$.
Since $\mu$ attains $(q-1)/2$ different values, $\mu$ must take different values
on the $(q-1)/2$ orbits of $\PGL_2(\BF_q)$ acting on products of irreducible
quadratics.

One can show that $\mu$ is derived from a modular function that parametrizes
pairs $(E,P)$, where $E$ is an elliptic curve and $P$ is a point of order~$2$. 
For products $f_1 f_2$ of two irreducible quadratics, the elliptic curve is the
Jacobian of the curve $C\colon z^2 = f_1 f_2$, and the point of order~$2$ is
represented by the degree-$0$ divisor on $C$ whose double is the divisor 
of~$f_1/f_2$.
\end{remark}

We also have a similar result for quartics with exactly one irreducible
quadratic factor, which works in all characteristics.

\begin{theorem}
\label{T:211}
Given a prime power $q$, let $\zeta$ be a generator of $\BF_{q^2}^\times$. Let 
\[
S_{211} = \{ \zeta^i \ \vert \ 0<i\le (q+1)/2 \}
\]
and let $T_{211}$ be the set of homogenized minimal polynomials of the elements
of $S_{211}$. Then the set $\{ xyg \ \vert\ g \in T_{211}\}$ is a complete set
of unique representatives for the action of $\PGL_2(\BF_q)$ on the monic 
separable homogeneous quartics that have exactly one irreducible quadratic
factor.
\end{theorem}

\begin{proof}
The proof follows the same lines as that of Theorem~\ref{T:twoquadlist}, but is
much simpler. We move the two rational roots of $f$ to $0$ and $\infty$, and
then show that up to scaling and inversion, every element of 
$\BF_{q^2}\setminus\BF_q$ has a unique representative in $S_{211}$. We leave the
details to the reader.
\end{proof}

\begin{remark}
The function $\mu$ defined in Remark~\ref{R:mu} can be extended to monic
separable quartics with exactly one irreducible quadratic factor; when the two
zeros of $f$ in $\BP^1(\BF_q)$ are finite we can use the same formula as before,
and when $f = y(x-by)(x^2 + ux + v)$ we can define $\mu(f) = (u+2b)^2/(u^2-4v)$.
Once again, $\mu$ depends only on the $\PGL_2(\BF_q)$ orbit of its argument. On
quartics of this type, the set of values attained by $\mu$ is the set of
nonsquares in $\BF_q$ together with~$0$. Thus, $\mu$ distinguishes
$\PGL_2(\BF_q)$ orbits of such quartics from one another.
\end{remark}

Theorem~\ref{T:211} applies to all finite fields, while 
Theorems~\ref{T:quartics} and~\ref{T:twoquadlist} require the characteristic to
be odd. The following theorem generalizes the latter two results to 
characteristic~$2$. The proof is analogous to those of the earlier theorems, but
is made much simpler by the fact that in this case every involution in 
$\PGL_2(\BF_q)$ is conjugate to $\twobytwo{1}{1}{0}{1}$. We leave the details to
the reader.

\begin{theorem}
\label{T:quarticschar2}
Let $q$ be a power of~$2$, let $A$ be the set of elements of $\BF_q$ of absolute
trace~$1$, and let $\nu$ be an element of $A$. Then the set
\[
\bigl\{ (x^4 + x^2y^2) + a(x^2y^2 + xy^3) + a^2\nu y^4\ \vert\ a \in A\bigr\}
\]
is a complete set of unique representatives for the action of $\PGL_2(\BF_q)$
on the monic irreducible homogeneous quartics over $\BF_q$, and the set
 \[
\bigl\{ (x^2 + x y + \nu y^2) (x^2 + x y + a y^2) \ \vert \ a\in A, a\ne\nu \bigr\}
 \]
is a complete set of unique representatives for the action of $\PGL_2(\BF_q)$
on the homogeneous quartics over $\BF_q$ that can be written as the
product of two distinct monic irreducible quadratics.\qed
\end{theorem}

\begin{remark}
\label{R:invertmu}
Theorems~\ref{T:quartics}, \ref{T:twoquadlist}, \ref{T:211}, 
and~\ref{T:quarticschar2} lead to quasilinear-time algorithms to give complete 
sets of unique representatives for the action of $\PGL_2(\BF_q)$ on the various
types of quartics discussed in the theorems. The only difficulty is obtaining a 
primitive element $\zeta$ for $\BF_{q^2}$ in Theorems~\ref{T:twoquadlist}
and~\ref{T:211} and an element $\gamma$ of order $2(q^2-1)$ in $\BF_{q^4}$ for
Theorem~\ref{T:quartics}. But primitive elements for $\BF_{q^2}$ can be
determined deterministically in time $O(q^{1/2+\varepsilon})$ for every
$\varepsilon>0$ (see~\cite{Shparlinski1996}), and the $\gamma$ required for
Theorem~\ref{T:quartics} can be obtained by taking the square root in 
$\BF_{q^4}$ of a primitive element for~$\BF_{q^2}$. 

In quasilinear time we can also create a table of size $O(q)$ that we can
use to invert the function~$\mu$, which gives us a way to compute the orbit
representatives of quartics with one or two irreducible quadratic factors. The
function $\mu$ does not reduce well modulo~$2$, but the function $(\mu-1)/4$
does; the corresponding function takes a product 
$(x^2 + sxy + ty^2)(x^2 + uxy + vy^2)$ in characteristic~$2$ to 
$((t+v)^2 + (s+u)(sv+tu))/(su)^2$.
\end{remark}

\section{An invariant for irreducible polynomials over finite fields}
\label{S:invariant}

In this section we define an easily computable invariant\footnote{Classically,
  an \emph{invariant} on the set $R_n\subset k[x,y]$ of homogeneous bivariate
  polynomials of degree~$n$ is a function $R_n\to k$ that is constant on 
  $\PGL_2(k)$ orbits. We use the term more generally here, and simply mean a 
  function from $R_n$ to \emph{any} set that is constant on $\PGL_2(k)$ orbits.}
for monic irreducible homogeneous bivariate polynomials of arbitrary degree 
$n\ge 4$ over a finite field $\BF_q$ under the action of $\PGL_2(\BF_q)$
mentioned in Section~\ref{S:Weierstrass}. The invariant is based on the 
classical \emph{cross ratio}, which is the function that assigns to an ordered 
quadruple $(P_1,P_2,P_3,P_4)$ of distinct elements of $\BP^1(\kbar)$ the element
$\alpha\in\kbar$ for which ${\Gamma(P_4) = \projtwo{\alpha}{1}}$, where $\Gamma$
is the unique element of $\PGL_2(\kbar)$ that sends $P_1$ to~$\infty$, $P_2$
to~$0$, and $P_3$ to~$1$. It follows that the cross ratio is constant on the 
orbits of the diagonal action of $\PGL_2(\kbar)$ on such quadruples, and takes
distinct values on distinct orbits. (Compare to 
\cite[Definition~III.3.7]{Conway1978} and the propositions following it.)

\begin{definition}
Let $q$ be a prime power and let $f$ be a monic irreducible homogeneous 
bivariate polynomial of degree $n\ge 4$ over~$\BF_q$. We define the 
\emph{cross polynomial} $\Cross(f)$ of $f$ as follows: Let $\alpha\in\BF_{q^n}$
be a root of~$f$, and let $\chi\in\BF_{q^n}$ be the cross ratio of $\alpha$,
$\alpha^q$, $\alpha^{q^2}$, and $\alpha^{q^3}$; that is,
\[
\chi\colonequals
\frac{(\alpha^{q^3}-\alpha^q)(\alpha^{q^2}-\alpha)}
     {(\alpha^{q^3}-\alpha)(\alpha^{q^2}-\alpha^{q})}\,.
\]
Then set $\Cross(f)$ to be the characteristic polynomial of $\chi$ over~$\BF_q$.
\end{definition}

Note that the denominator of $\chi$ is nonzero, because the powers of $\alpha$
involved are four of the $n$ distinct conjugates of $\alpha$. Also, replacing
$\alpha$ with one of its conjugates results in replacing $\chi$ with a
conjugate, so the characteristic polynomial remains unchanged. Thus we see that
$\Cross(f)$ is well-defined.

\begin{theorem}
\label{T:invariant}
Let $q$ be a prime power. Two monic irreducible homogenous polynomials in 
$\BF_q[x,y]$ of degree at least~$4$ lie in the same orbit under the action of
$\PGL_2(\BF_q)$ if and only if they have the same cross polynomial.
\end{theorem}

\begin{proof}
Suppose $f$ and $g$ are irreducible homogenous polynomials in $\BF_q[x,y]$ of
degree at least~$4$. Suppose $f$ and $g$ lie in the same $\PGL_2(\BF_q)$ orbit,
say $g = \Gamma(f)$ for some $\Gamma\in\PGL_2(\BF_q)$. Then $f$ and $g$ have the
same degree, which we denote by $n$. Let $\alpha$ be a root of $f$ in 
$\BF_{q^n}$ and let $\beta = \Gamma(\alpha)$. Then $\beta$ is a root of~$g$, and
for every $i\ge 0$ we have $\beta^{q^i} = \Gamma(\alpha^{q^i})$. In particular, 
the cross ratio of $\alpha$, $\alpha^q$, $\alpha^{q^2}$, and $\alpha^{q^3}$ is
equal to the cross ratio of  $\beta$, $\beta^q$, $\beta^{q^2}$, and 
$\beta^{q^3}$, so $\Cross(f) = \Cross(g)$.

Conversely, suppose $f$ and $g$ are two monic irreducible polynomials of degree
at least $4$ with $\Cross(f) = \Cross(g)$. Since a polynomial has the same
degree as its cross polynomial, $f$ and $g$ have the same degree, say~$n$. Since
the cross polynomials are equal, there are roots $\alpha$ of $f$ and $\beta$ of
$g$ in $\BF_{q^n}$ such that $\alpha$, $\alpha^q$, $\alpha^{q^2}$, and
$\alpha^{q^3}$ have the same cross ratio as $\beta$, $\beta^q$, $\beta^{q^2}$,
and $\beta^{q^3}$. It follows that there is an element $\Gamma$ of 
$\PGL_2(\BF_{q^n})$ with $\Gamma(\alpha^{q^i}) = \beta^{q^i}$ for $0\le i\le 3$.
In particular, we have $\Gamma(x^q)=\Gamma(x)^q$ for three distinct values 
of~$x$, namely $\alpha$, $\alpha^q$, and $\alpha^{q^2}$, so $\Gamma$ is fixed by
Frobenius and therefore lies in $\PGL_2(\BF_q)$. Thus, $\Gamma$ takes every root
of $f$ to a root of $g$, so $\Gamma(f) = g$.
\end{proof}

As an application of this invariant, we give an algorithm for creating a table
of orbit representatives for irreducible polynomials of degree $n\ge 4$ in 
time~$\Otilde(q^{n-2})$.

\begin{algorithm}
\label{A:naivelist}
Inverting the cross polynomial function.
\begin{alglist}
\algin  A prime power~$q$ and an integer $n \ge 4$.
\algout A table, indexed by the values of the cross polynomials for irreducible
        polynomials of degree~$n$, giving for each cross polynomial $g$ an
        irreducible homogeneous $f\in \BF_q[x,y]$ of degree~$n$ with 
        $\Cross f = g$.
\item \label{naivelistinitialize}
      Construct a copy of $\BF_{q^n}$ with an $\BF_q$-basis 
      $(\beta_1,\ldots,\beta_n)$ such that $\beta_1$ appears with nonzero
      coefficient in the representation of~$1$.
\item Set $L$ to be the empty list.
\item \label{naivelistalpha}
      For every $\alpha\in\BF_{q^n}$ that does not lie in a proper subfield, and
      whose representation $(a_1,\ldots, a_n)$ on the given basis has $a_1=0$
      and has $a_i = 1$ for the first $i$ with $a_i\ne 0$, do:
      \begin{algsublist}
      \item Compute the homogenization $f$ of the minimal polynomial
            of~$\alpha$.
      \item Compute $\Cross f$.
      \item Append the pair $(\Cross f, f)$ to $L$.
      \end{algsublist}
\item Sort $L$.
\item \label{naivelistdelete}
      Delete every entry $(\Cross f, f)$ of $L$ for which the value of
      $\Cross f$ appears earlier in the list.
\item Return $L$.      
\end{alglist}
\end{algorithm}

\begin{proposition}
Algorithm \textup{\ref{A:naivelist}} produces correct output and runs in time 
$\Otilde(q^{n-2})$, measured in arithmetic operations in~$\BF_q$.
\end{proposition}

\begin{proof}
First we note that every $\PGL_2(\BF_q)$ orbit contains an $\alpha$ as in 
step~\eqref{naivelistalpha}. We can see this because starting with an arbitrary 
$\alpha\in \BF_{q^n}$, we can subtract an element of $\BF_q$ to zero out the 
coefficient of~$\beta_1$, and then we can scale by an element of $\BF_q^\times$
so that the first nonzero coefficient is $1$. It follows that every orbit will
have a representative included in the output, and step~\eqref{naivelistdelete}
ensures that there is only one representative given for each orbit. Thus the
output is correct. Now we analyze the timing.

For fixed $n$, Shoup's algorithm~\cite{Shoup1990} can construct a finite field
$\BF_{q^n}$ in time~$\Otilde(\sqrt{q})$, and in polynomial 
time~\cite{Lenstra1991} we can find an embedding of our given $\BF_q$ into this 
copy of~$\BF_{q^n}$, so step~\eqref{naivelistinitialize} can be done within the
stated time bound. Because $n$ is fixed, for each $\alpha$ the values of $f$ and
$\Cross f$ can be computed in time $O(1)$, so creating the list $L$ takes time
$O(q^{n-2})$. Finally, sorting a list of length $O(q^{n-2})$ takes time 
$\Otilde(q^{n-2})$ (see~\cite[\S5.2.3]{Knuth1998}).
\end{proof}

As we will see, for composite values of $n$ there is an algorithm for producing
a complete set of unique representatives for the $\PGL_2(k)$ orbits of 
irreducible homogeneous polynomials of degree~$n$ that runs in time 
$\Otilde(q^{n-3})$; see Section~\ref{SS:PGL2orbits}. An algorithm of this time
complexity that works for all $n$ is given in~\cite{Howe2024}.

\section{Explicit coset representatives for
         \texorpdfstring{$\PGL_2(\BF_q)$}{PGL2 Fq}
         in \texorpdfstring{$\PGL_2(\BF_{q^2})$}{PGL2 Fq2}}
\label{S:CosetRepsPGL2}

As part of our algorithm, we will need to have a complete set of unique 
representatives for the right cosets of the subgroup $\PGL_2(\BF_q)$ of 
$\PGL_2(\BF_{q^2})$. In this section we give an explicit set of such
representatives. 

Throughout this section, $q$ is a prime power, $\omega$ is an element of 
${\BF_{q^2}\setminus\BF_q}$, and $\gamma$ is a generator of the multiplicative
group of $\BF_{q^2}$.

An element of $\PGL_2(\BF_{q^2})$ is determined by where it sends $\infty$, $0$, 
and $1$, and given any three distinct elements of $\BP^1(\BF_{q^2})$, there is
an element of $\PGL_2(\BF_{q^2})$ that sends $\infty$, $0$, and $1$ to those
three elements. Thus, we may represent elements of $\PGL_2(\BF_{q^2})$ by 
triples $(\zeta,\eta,\theta)$ of pairwise distinct elements of 
$\BP^1(\BF_{q^2})$, indicating the images of $\infty$, $0$, and $1$. If $\Gamma$
is an element of $\PGL_2(\BF_q)$, then $\Gamma$ sends the element 
$(\zeta,\eta,\theta)$ of $\PGL_2(\BF_{q^2})$ to 
$(\Gamma(\zeta), \Gamma(\eta),\Gamma(\theta))$.

\begin{proposition}
\label{P:CosetRepsPGL2}
Let $B$ be the set 
$\{(\omega\gamma^i + \omega^q)/(\gamma^i+1) \ \vert \ 0\le i < q-1\bigr\}.$ The
following elements give a complete set of unique coset representatives for the 
left action of $\PGL_2(\BF_q)$ on $\PGL_2(\BF_{q^2})$\textup{:}
\begin{enumerate}
\item \label{CRPGL21}
      $(\infty, 0, 1)$\textup{;}
\item \label{CRPGL22}
      $\bigl\{ (\infty, 0, \omega + a) \ \vert \ a \in \BF_q \bigr\}$\textup{;}
\item \label{CRPGL23}
      $\bigl\{ (\infty, \omega, \theta) \ \vert \ \theta \in \BF_{q^2} \text{\ with\ } \theta \neq \omega \bigr\}$\textup{;}
\item \label{CRPGL24}
      $\bigl\{ (\omega, \omega^q, \theta) \ \vert \ \theta\in B \bigr\}$\textup{;}
\item \label{CRPGL25}
      $\bigl\{ (\omega, \eta, \theta) \ \vert \ \eta\in B,  \theta \in \BP^1(\BF_{q^2})
      \text{\ with\ } \theta \neq \omega \text{\ and\ }\theta\ne \eta \bigr\}$.
\end{enumerate}
\end{proposition}

To prove this proposition, we need the following lemma.

\begin{lemma}
\label{L:OrbitRepsG}
With notation as above, let $G$ be the subgroup of $\PGL_2(\BF_q)$ that fixes
the element $\omega$. Then the set $B$ from 
Proposition~\textup{\ref{P:CosetRepsPGL2}} is a complete set of unique 
representatives for the left action of $G$ on 
${\BP^1(\BF_{q^2})\setminus\{\omega,\omega^q\}}$.
\end{lemma}

\begin{proof}
Let $r \colonequals \omega+\omega^q$ and let $s\colonequals \omega^{q+1}$. We 
check that the group $G$ is equal to 
\[
G = \Biggl\{ \begin{bmatrix}a & \phantom{a} - s b\\ b & a - r b \end{bmatrix}
             \ \bigg\vert \ \projtwo{a}{b} \in \BP^1(\BF_q) \Biggr\}\,.
\]
Let 
$\Phi\colonequals \twobytwo{-1}{\phantom{-}\omega^q}{\phantom{-}1}{-\omega\phantom{^q}},$
so that $\Phi$ sends $\omega$ to $\infty$ and $\omega^q$ to~$0$. We compute that
\begin{align*}
\Phi G \Phi^{-1} 
&= 
\Biggl\{ \begin{bmatrix} a - b\omega^q&0\\0 & a - b\omega\end{bmatrix} 
             \ \bigg\vert \  \projtwo{a}{b} \in \BP^1(\BF_q) \Biggr\}\\
&= 
\Biggl\{ \begin{bmatrix} (a - b\omega^q)/(a-b\omega)&0\\0 & 1\end{bmatrix} 
             \ \bigg\vert \  \projtwo{a}{b} \in \BP^1(\BF_q) \Biggr\}\,.             
\end{align*}
By Hilbert~$90$, the set of values attained by $(a - b\omega^q)/(a - b\omega)$ 
is equal to the set of elements of $\BF_{q^2}$ whose norms to $\BF_q$ are equal 
to~$1$, and these elements are precisely the powers of $\gamma^{q-1}$. Thus, the
action of $\Phi G\Phi^{-1}$ on $\BP^1(\BF_{q^2})\setminus\{\infty,0\}$ is 
generated by multiplication by $\gamma^{q-1}$, and it is easy to see that the 
values $1, \gamma, \ldots, \gamma^{q-2}$ are orbit representatives for this 
action. Applying $\Phi^{-1}$ to these orbit representatives will give us orbit
representatives for the action of $G$ on 
${\BP^1(\BF_{q^2})\setminus\{\omega,\omega^q\}}$, and we see that 
$\Phi^{-1}(\gamma^i) = (\omega\gamma^i+\omega^q)/(\gamma^i + 1)$. 
\end{proof}

\begin{proof}[Proof of Proposition~\textup{\ref{P:CosetRepsPGL2}}]
Suppose we are given an element $(\zeta,\eta,\theta)$ of $\PGL_2(\BF_{q^2})$. We
will show how to modify it by elements of $\PGL_2(\BF_q)$ to put it into one of 
the forms listed in the proposition. In the course of this demonstration, it
will become clear that the elements listed in the proposition do indeed lie in 
different $\PGL_2(\BF_q)$ orbits, because they are fixed by the following 
procedure.

Recall that $\omega$ is an element of $\BF_{q^2}\setminus\BF_q$. Given a triple
$\Gamma\colonequals (\zeta,\eta,\theta)$ representing an element of
$\PGL_2(\BF_{q^2})$, we do the following:

\begin{enumerate}
\item \emph{If $\zeta$ lies in $\BP^1(\BF_q)$}: 
      In this case, we can apply an element of $\PGL_2(\BF_q)$ that moves
      $\zeta$ to $\infty$. Our element of $\PGL_2(\BF_{q^2})$ can now be written
      $(\infty, \eta, \theta)$, for some new values of $\eta$ and $\theta$. We 
      can now only apply elements of $\PGL_2(\BF_q)$ that fix $\infty$; that is,
      we are limited to the so-called $ax+b$ group.
      \begin{enumerate}
      \item \emph{If $\eta$ lies in $\BF_q$}:
            In this case, we can use the $ax+b$ group to move $\eta$ to~$0$. Our
            element of $\PGL_2(\BF_{q^2})$ can now be written 
            $(\infty, 0, \theta)$, for some new value $\theta$. Now we can only 
            apply elements of $\PGL_2(\BF_q)$ that fix $\infty$ and~$0$; that
            is, we are limited to scalar multiplication.
            \begin{enumerate}
            \item \emph{If $\theta$ lies in $\BF_q$}:
                  In this case, we can scale $\theta$ so that it is equal
                  to~$1$. We obtain the element $(\infty,0,1)$ listed in part
                  \eqref{CRPGL21} of the proposition, and no further action of 
                  $\PGL_2(\BF_q)$ is possible.
            \item \emph{If $\theta$ does not lie in $\BF_q$}:
                  We can write $\theta = u\omega + v$ for elements $u,v$ of 
                  $\BF_q$, with $u$ nonzero. There is a unique scaling that will
                  put $\theta$ into the form $\omega + a$. We obtain an element
                  from part \eqref{CRPGL22} of the proposition.
            \end{enumerate}
      \item \emph{If $\eta$ does not lie in $\BF_q$}:
            Using the $ax+b$ group, we can move $\eta$ to $\omega$. There is no 
            further action of $\PGL_2(\BF_q)$ that fixes $\infty$ and $\omega$,
            so $\theta$ can be any element of $\BF_{q^2}$ other than $\omega$. 
            This gives us an element from part \eqref{CRPGL23} of the 
            proposition.
      \end{enumerate}
\item \emph{If $\zeta$ does not lie in $\BP^1(\BF_q)$}: 
      In this case, $\zeta$ is an element of $\BF_{q^2}\setminus\BF_{q}$, and we
      can use the $ax+b$ subgroup of $\PGL_2(\BF_q)$ to move $\zeta$ 
      to~$\omega$. The only elements of $\PGL_2(\BF_q)$ that we can apply once
      we have fixed $\zeta = \omega$ are the elements of the group $G$ from 
      Lemma~\ref{L:OrbitRepsG}.
      \begin{enumerate}
      \item \emph{If $\eta$ is equal to $\omega^q$}:
            If $\eta=\omega^q$ then the action of $G$ fixes $\eta$. We know that
            $\theta$ is different from both $\omega$ and $\omega^q$, so by 
            Lemma~\ref{L:OrbitRepsG} we can use $G$ to move $\theta$ to a unique
            element of the set $B$. This gives us an element from part 
            \eqref{CRPGL24} of the proposition.
      \item \emph{If $\eta$ is not equal to $\omega^q$}:
            We can use $G$ to move $\eta$ to a unique element of~$B$. Once we 
            have normalized $\eta$ in this way, there is no further action of 
            $\PGL_2(\BF_q)$ that fixes $\omega$ and~$\eta$, so $\theta$ can be
            any element of $\BP^1(\BF_{q^2})$ other than $\omega$ and $\eta$. 
            This gives us an element from part \eqref{CRPGL25} of the
            proposition.
      \end{enumerate}  
\end{enumerate}
These cases enumerate all of the possibilities for our element 
$(\zeta,\eta,\theta)$, so the proposition is proved.
\end{proof}

\section{Explicit coset representatives for 
         \texorpdfstring{$\PGL_2(\BF_q)$}{PGL2 Fq}
         in \texorpdfstring{$\PGL_2(\BF_{q^p})$}{PGL2 Fqp}}
\label{S:CosetRepsPGLp}

It is not necessary for proving our main theorem, but there is a result
analogous to Proposition~\ref{P:CosetRepsPGL2} for the cosets of $\PGL_2(\BF_q)$
in $\PGL_2(\BF_{q^p})$, where $p$ is an odd prime. As before, we represent
elements of $\PGL_2(\BF_{q^p})$ as triples $(\zeta,\eta,\theta)$ of distinct
elements of $\BP^1(\BF_{q^p})$, indicating where the given element of 
$\PGL_2(\BF_{q^p})$ sends $\infty$, $0$, and $1$.

\begin{proposition}
\label{P:CosetRepsPGLp}
Let $q$ be a prime power and let $p$ be an odd prime. Let $C$ be a set of orbit
representatives for the action of $\PGL_2(\BF_q)$ on $\BF_{q^p}\setminus\BF_q$, 
let $C_{\infty} $ be a set of orbit representatives for the action of the $ax+b$
subgroup of $\PGL_2(\BF_q)$ on $\BF_{q^p}\setminus\BF_q$, and let $C_{\infty,0}$
be a set of orbit representatives for the multiplicative action of
$\BF_q^\times$ on $\BF_{q^p}\setminus\BF_q$. 

The following elements give a complete set of unique coset representatives for 
the left action of $\PGL_2(\BF_q)$ on $\PGL_2(\BF_{q^p})$\textup{:}
\begin{enumerate}
\item \label{CRPGLp1}
      $(\infty, 0, 1)$\textup{;}
\item \label{CRPGLp2}
      $\bigl\{ (\infty, 0, \theta) 
       \ \vert \ \theta\in C_{\infty,0}\bigr\}$\textup{;}
\item \label{CRPGLp3}
      $\bigl\{ (\infty, \eta, \theta) 
       \ \vert \ \eta\in C_{\infty}, \ \theta \in \BF_{q^p} 
       \text{\ with\ } \theta \neq \eta \bigr\}$\textup{;}
\item \label{CRPGLp4}
      $\bigl\{ (\zeta, \eta, \theta) 
       \ \vert \ \zeta\in C,  \ \eta, \theta \in \BP^1(\BF_{q^p})
       \text{\ with\ } \eta\neq \zeta 
       \text{\ and\ } \theta \neq \zeta 
       \text{\ and\ } \theta\ne\eta \bigr\}$.
\end{enumerate}
\end{proposition}

\begin{proof}
The proof is much like that of Proposition~\ref{P:CosetRepsPGL2}, but simpler.
Suppose we are given an arbitrary $(\zeta,\eta,\theta)$ in $\PGL_2(\BF_{q^p})$.
If $\zeta$ and $\eta$ both lie in $\BP^1(\BF_q)$, we can move them to $\infty$
and $0$, and then we can only modify $\theta$ by scaling by elements 
of~$\BF_q^\times$. If $\theta$ lies in $\BF_q$ we get case~\eqref{CRPGLp1}, and
if not we get case~\eqref{CRPGLp2}.

If $\zeta$ lies in $\BP^1(\BF_q)$ but $\eta$ does not, we move $\zeta$ to 
$\infty$ using $\PGL_2(\BF_q)$. Then the only action of $\PGL_2(\BF_q)$ we have 
left to us is the $ax+b$ subgroup. Since $\eta$ lies in 
$\BF_{q^p}\setminus\BF_q$, we can use this subgroup to move $\eta$ to an element
of $C_\infty$. Then $\theta$ can be arbitrary, as long as it is different from
$\infty$ and from $\eta$. This gives us case~\eqref{CRPGLp3}.

If $\zeta$ does not lie in $\BP^1(\BF_q)$ then we can move $\zeta$ using 
$\PGL_2(\BF_q)$ so that it lies in $C$, and there is no further action of 
$\PGL_2(\BF_q)$ left available to us, because the only elements of 
$\BP^1(\BFbar_q)$ with nontrivial $\PGL_2(\BF_q)$ stabilizers lie in 
$\BP^1(\BF_{q^2})$. This gives us case~\eqref{CRPGLp4}.
\end{proof}

This result is useful because we can compute the sets of representatives we 
need. 

\begin{algorithm}
\label{A:PGLreps}
Orbit representatives for $\PGL_2(\BF_q)$ acting on the elements of $\BF_{q^n}$ 
that do not lie in proper subfields.
\begin{alglist}
\algin  A prime power $q$ and an integer $n\ge 3$.
\algout A complete set of unique representatives for the action of 
        $\PGL_2(\BF_q)$ on the the elements of $\BF_{q^n}$ that do not lie in
        proper subfields.
\item Set $L$ and $M$ to be empty lists.
\item Construct a copy of $\BF_{q^n}$ with an $\BF_q$-basis 
      $(\beta_1,\ldots,\beta_n)$ such that $\beta_1$ appears with nonzero 
      coefficient in the representation of~$1$.
\item \label{PGLreps3}
      If $n=3$ return a list containing the single element $\beta_2$, and stop.      
\item For every $\alpha\in\BF_{q^n}$ that does not lie in a proper subfield, and
      whose representation $(a_1,\ldots, a_n)$ on the given basis has $a_1=0$
      and has $a_i = 1$ for the first $i$ with $a_i\ne 0$, do:
      \begin{algsublist}
      \item Compute $\alpha^{q^i}$ for $i=1,\ldots,n-1$. If any of these
            conjugates is smaller than or equal to $\alpha$ under a fixed
            ordering $<$, continue on to the next value of $\alpha$.
      \item Compute the minimal polynomial $f$ of $\alpha$.
      \item Find the (unique) irreducible factor $g$ of $\Cross(f)$.
      \item Append the pair $(g, \alpha)$ to $L$.
      \end{algsublist}
\item Sort $L$.
\item Delete every element $(g, \alpha)$ of $L$ such that $g$ appears as a first
      entry of an element earlier in the list.
\item For every $(g,\alpha)$ in $L$, do:
      \begin{algsublist}
      \item For $i = 0,\ldots, \deg g - 1$, append the element $\alpha^{q^i}$ 
            to~$M$.
      \end{algsublist}
\item Return $M$.
\end{alglist}
\end{algorithm}

\begin{proposition}
Algorithm~\textup{\ref{A:PGLreps}} produces a complete list of unique 
representatives for the orbits of $\PGL_2(\BF_q)$ acting on the elements of 
$\BF_{q^n}$ that lie in no proper subfield. It runs in time $\Otilde(q^{n-2})$.
\end{proposition}

\begin{proof}
When $n=3$, the group $\PGL_2(\BF_q)$ acts transitively on
$\BP_{q^3}\setminus\BF_q$, so step~\eqref{PGLreps3} gives correct output. For 
$n>3$, Algorithm~\ref{A:PGLreps} is a variation on Algorithm~\ref{A:naivelist}.
The only additional fact we must note is that there is an element of 
$\PGL_2(\BF_q)$ that takes $\alpha$ to one of its nontrivial conjugates if and
only if the cross polynomial of $f$ is not irreducible, and that the order of
each such element of $\PGL_2(\BF_q)$ is equal to the exponent $e$ such that
$\Cross(f) = g^e$. Thus, the Galois orbit of $\alpha$ contains representatives
of exactly $\deg g$ $\PGL_2(\BF_q)$ orbits.
\end{proof}

Algorithm~\ref{A:PGLreps} gives us a method to calculate the set $C$ from
Proposition~\ref{P:CosetRepsPGLp}. The sets $C_\infty$ and $C_{\infty,0}$ can be
computed in similar (but simpler) ways; we leave the details to the reader.

\section{Enumerating \texorpdfstring{$\PGL_2(\BF_q)$}{PGL2 Fq} orbit 
         representatives for \texorpdfstring{$\Sym^n(\BF_q)$}{Sym n Fq}}
\label{S:enumerating}

In this section we present our algorithm for enumerating orbit representatives
for the action of $\PGL_2(\BF_q)$ on $\Sym^n(\BF_q)$ in time $\Otilde(q^{n-3})$,
for $n$ fixed and $q$ varying. The algorithm consists of a number of different
algorithms, each addressing a subset of elements of $\Sym^n(\BF_q)$. The problem
is trivial when $n\le 3$, so we will always assume that $n\ge 4$, and for 
one case we will also demand that $n$ be even. This is sufficient for our 
application to enumerating hyperelliptic curves of genus~$g$, where $n = 2g+2$
is even and at least~$6$. (See~\cite{Howe2024} for an algorithm that works for
all~$n$.)

Recall that an element of $\Sym^n(\BF_q)$ is a set 
$A = \{ \alpha_1, \ldots, \alpha_n\}$ of $n$ distinct elements of
$\BP^1(\BFbar_q)$ that is stable under the action of the Galois group of 
$\BFbar_q$ over~$\BF_q$. An element of $\Sym^n(\BF_q)$ is \emph{primitive} if 
the degree of each extension $\BF_q(\alpha_i)$ over $\BF_q$ is equal to~$n$.
Every element $A$ of $\Sym^n(\BF_q)$ can be written in a unique way (up to
order) as the union of a collection $\{A_i\}$ of primitive elements $A_i$ of 
$\Sym^{m_i}(\BF_q)$, for some sequence of integers $m_i$ with $\sum m_i = n$.
The sequence $(m_i)_i$, listed in non-increasing order, is the 
\emph{Galois type} of~$A$. If $f$ is the monic homogeneous polynomial whose zero
set is~$A$, then $(m_i)_i$ is also the list of the degrees of the irreducible 
factors of $f$, and we also refer to this sequence as the Galois type of~$f$. We
will enumerate the $\PGL_2(\BF_q)$ orbits of $\Sym^n(\BF_q)$ by enumerating each
Galois type separately.

Let $M\colonequals (m_1, m_2, \ldots, m_r)$ be a Galois type for 
$\Sym^n(\BF_q)$, so that $m_1\ge m_2\ge \cdots \ge m_r>0$ and 
$n = m_1 + \cdots + m_r$. In the following subsections we show how to enumerate
the $\PGL_2(\BF_q)$ orbits of the elements of $\Sym^n(\BF_q)$ of this Galois
type, based on the value of $m_1$.

\subsection{The case \texorpdfstring{$m_1 = 1$}{m1 = 1}} 
Every element $A$ of $\Sym^n(\BF_q)$ of this Galois type is simply a collection
of $n$ distinct elements of $\BP^1(\BF_q)$. We can specify a standard form for 
such elements $A$ by considering all possible choices of three distinct points 
$a_i$, $a_j$, and $a_k$ in~$A$, and using an element $\Gamma$ of $\PGL_2(\BF_q)$
to move those three points to $\infty$, $0$, and $1$, respectively. To this
choice we associate the polynomial $f$ of degree $n-1$ defined by
\[
f\colonequals y \prod_{\ell\ne i} (x - \Gamma(a_\ell)y)\,.
\]
Our standard form for $A$ is the smallest polynomial $f$ obtained in this way,
under an arbitrary total ordering $<$ of the monic homogeneous polynomials of 
degree~$n$.

Our algorithm for enumerating orbit representatives of this Galois type is as
follows. 

\begin{algorithm}
\label{A:1}
Orbit representatives for $\PGL_2(\BF_q)$ acting on the elements of
$\Sym^n(\BF_q)$ of Galois type $(1,1,\ldots,1)$.
\begin{alglist}
\algin  A prime power $q$ and an integer $n\ge 4$.
\algout A complete set of unique representatives for the action of 
        $\PGL_2(\BF_q)$ on the monic homogenous polynomials of degree~$n$ and 
        Galois type $(1,1,\ldots,1)$.
\item Set $L$ to be the empty list, and set $a_1\colonequals\infty$, 
      $a_2\colonequals 0$, and $a_3\colonequals 1$.
\item \label{1dedupe}
      For every set $\{a_4,\ldots,a_n\}$ of distinct elements of 
      $\BF_q\setminus\{0,1\}$ do:
      \begin{algsublist}
      \item Set $f\colonequals y \prod_{i=2}^n (x-a_i y)$.
      \item Set $F\colonequals\{\Gamma(f)\}$, where $\Gamma$ ranges over the 
            elements of $\PGL_2(\BF_q)$ that send three elements of $\{a_i\}$ to
            $\infty$, $0$, and~$1$.
      \item If $f$ is the smallest element of $F$ under the ordering $<$, append
            $f$ to~$L$.
      \end{algsublist}
\item Return $L$.
\end{alglist}
\end{algorithm}

\begin{proposition}
Algorithm~\textup{\ref{A:1}} produces a complete set of unique representatives
for the orbits of $\PGL_2(\BF_q)$ acting on the monic homogeneous degree-$n$
polynomials of Galois type $(1,1,\ldots,1)$. It runs in time $\Otilde(q^{n-3})$,
measured in arithmetic operations in $\BF_q$. \qed
\end{proposition}

\subsection{The case \texorpdfstring{$m_1 = 2$}{m1 = 2}} 
When $m_1=2$ the possible Galois types $(m_1,\ldots,m_r)$ consist of $s$ values 
of $2$ and $t$ values of $1$, where $2s+t=n$ and $s>0$. We present two 
algorithms, one that applies when $t\ge 3$ and one that applies when $s\ge 2$.
Since we are assuming throughout that $n\ge 4$, the only remaining case
is when $s=1$ and $t=2$, but that situation is handled by~Theorem~\ref{T:211}.

\begin{algorithm}
\label{A:2a}
Orbit representatives for $\PGL_2(\BF_q)$ acting on the elements of 
$\Sym^n(\BF_q)$ of Galois type of the form $(2, 2,\ldots,2, 1, \ldots, 1)$, with 
$s$ entries of $2$ and $t$ entries of $1$, where $t\ge 3$.
\begin{alglist}
\algin  A prime power $q$, an integer $n\ge 4$, and integers $s$ and $t$
        with $2s+t=n$ and $t\ge 3$.
\algout A complete set of unique representatives for the action of 
        $\PGL_2(\BF_q)$ on the monic homogenous polynomials of degree~$n$ with
        the given Galois type.
\item Set $L$ to be the empty list, and set $a_1\colonequals\infty$,
      $a_2\colonequals 0$, and $a_3\colonequals 1$.
\item Create a list $I_2$ of the monic irreducible homogeneous quadratics 
      over~$\BF_q$.
\item \label{2adedupe}
      For every set $\{a_4,\ldots,a_t\}$ of distinct elements of 
      $\BF_q\setminus\{0,1\}$ and every set $\{g_1,\ldots, g_s\}$ of distinct
      elements of $I_2$ do:
      \begin{algsublist}
      \item Set $f\colonequals y \prod_{i=2}^t (x-a_i y)\cdot\prod_{i=1}^s g_i$.
      \item Set $F\colonequals\{\Gamma(f)\}$, where $\Gamma$ ranges over the 
            elements of $\PGL_2(\BF_q)$ that send three elements of $\{a_i\}$ to
            $\infty$, $0$, and $1$.
      \item If $f$ is the smallest element of $F$ under the ordering $<$, append
            $f$ to~$L$.
      \end{algsublist}
\item Return $L$.
\end{alglist}
\end{algorithm}

\begin{proposition}
Algorithm~\textup{\ref{A:2a}} produces a complete set of unique representatives
for the orbits of $\PGL_2(\BF_q)$ acting on the monic homogeneous degree-$n$ 
polynomials of Galois type $(2, 2,\ldots,2, 1, \ldots, 1)$, with $s$ entries of
$2$ and $t\ge 3$ entries of $1$. It runs in time $\Otilde(q^{n-3})$, measured in
arithmetic operations in $\BF_q$. \qed
\end{proposition}

\begin{algorithm}
\label{A:2b}
Orbit representatives for $\PGL_2(\BF_q)$ acting on the elements of 
$\Sym^n(\BF_q)$ of Galois type of the form $(2, 2,\ldots,2, 1, \ldots, 1)$, with
$s$ entries of $2$ and $t$ entries of $1$, where $s\ge 2$.
\begin{alglist}
\algin  A prime power $q$, an integer $n\ge 4$, and integers $s$ and $t$ 
        with $2s+t=n$ and $s\ge 2$.
\algout A complete set of unique representatives for the action of 
        $\PGL_2(\BF_q)$ on the monic homogenous polynomials of degree~$n$ with
        the given Galois type.
\item Set $L$ to be the empty list.
\item Let $I_1$ be the set of polynomials $\{y\}\cup\{ x-ay \,:\, a\in\BF_q\}$.
\item Create a list $I_2$ of the monic irreducible homogeneous quadratics 
      over~$\BF_q$.
\item \label{2bdedupe}
      For every pair $(f_1,f_2)$ of irreducible quadratic factors obtained from 
      Theorem~\ref{T:twoquadlist} or Theorem~\ref{T:quarticschar2}, for every
      set $\{f_3,\ldots, f_s\}$ of elements of $I_2$ such that the quadratics
      $f_1,\ldots, f_s$ are distinct, and for every set $\{g_1,\ldots, g_t\}$ of
      distinct elements of $I_1$, do:
      \begin{algsublist}
      \item Set $f\colonequals \prod_{i=1}^s f_i \cdot \prod_{i=1}^t g_i$.
      \item Set $F\colonequals\{\Gamma(f)\}$, where $\Gamma$ ranges over the
            elements of $\PGL_2(\BF_q)$ that send a pair of elements of 
            $\{ f_i\}$ to the representative of their $\PGL_2(\BF_q)$ orbit,
            calculated using the table mentioned at the end of 
            Remark~\ref{R:invertmu}.
      \item If $f$ is the smallest element of $F$ under the ordering $<$, append
            $f$ to~$L$.
      \end{algsublist}
\item Return $L$.
\end{alglist}
\end{algorithm}

\begin{remark}
\label{R:2broots}
The most direct way of implementing step~\eqref{2bdedupe}(b) involves computing
the roots of various irreducible quadratics in a fixed copy of $\BF_{q^2}$. 
From~\cite[Theorem~1.2]{Lenstra1991}, we know that this can be done in time
polynomial in $\log q$.
\end{remark}

\begin{proposition}
Algorithm~\textup{\ref{A:2b}} produces a complete set of unique representatives
for the orbits of $\PGL_2(\BF_q)$ acting on the monic homogeneous degree-$n$ 
polynomials of Galois type $(2, 2,\ldots,2, 1, \ldots, 1)$, with $s\ge 2$ 
entries of $2$ and $t$ entries of $1$. It runs in time $\Otilde(q^{n-3})$,
measured in arithmetic operations in $\BF_q$. \qed
\end{proposition}

\subsection{The case \texorpdfstring{$3\le m_1\le n-1$}{3 <= m1 <= n-1}} 
Let $(m_1,\ldots,m_r)$ be a Galois type with $3\le m_1\le n-1$. When $m_1>3$ we
will make use of the list of $\PGL_2(\BF_q)$ orbit representatives of 
irreducible polynomials of degree~$m_1$ provided by Algorithm~\ref{A:naivelist},
which will not exceed our claimed time bound of $\Otilde(q^{n-3})$ because
$m_1-2\le n-3$. When $m_1 = 3$ we will use the fact that there is exactly one 
$\PGL_2(\BF_q)$ orbit of irreducible degree-$3$ polynomials, so we can take our 
favorite irreducible polynomial of degree $3$ as the sole orbit representative.

\begin{algorithm}
\label{A:medium}
Orbit representatives for $\PGL_2(\BF_q)$ acting on the elements of
$\Sym^n(\BF_q)$ of Galois type $(m_1,\ldots,m_r)$, with $3\le m_1\le n-1$.
\begin{alglist}
\algin  A prime power $q$, an integer $n\ge 4$, and a Galois type with
        $3\le m_1\le n-1$.
\algout A complete set of unique representatives for the action of 
        $\PGL_2(\BF_q)$ on the monic homogenous polynomials of degree~$n$ with
        the given Galois type.
\item Set $L$ to be the empty list.
\item \label{mediumdegree}
      For each value of $m_i$ in the set $\{m_2,\ldots,m_r\}$, create a list
      $I_{m_i}$ of the monic irreducible homogeneous polynomials of degree~$m_i$.
\item \label{mediumlist}
      If $m_1>3$, let $L_1$ be the output of Algorithm~\ref{A:naivelist} 
      associated to the inputs $q$ and $m_1$ and let $S$ be the list of orbit
      representatives of irreducible polynomials of degree~$m_1$ obtained as the
      second elements of each pair on the list $L_1$.
\item If $m_1=3$ let $S$ be the single-element list consisting of an arbitrary
      irreducible polynomial of degree~$3$. 
\item \label{mediumdedupe}
      For every element $f_1$ of $S$ and every set of distinct polynomials
      $\{f_2, \ldots, f_r\}$ with $f_i\in I_{m_i}$ do:
      \begin{algsublist}
      \item Set $f\colonequals \prod_{i=1}^r f_i.$
      \item Let $M_1$ be the set of $f_i$ of degree $m_1$ and set 
            $F\colonequals\{\Gamma(f)\}$, where $\Gamma$ ranges over the
            elements of $\PGL_2(\BF_q)$ that send an element of $M$ to its
            associated orbit representative, obtained by computing its cross
            polynomial and using the lookup table $L_1$ if $m_1>3$ and by direct
            calculation if $m_1=3$.
      \item If $f$ is the smallest element of $F$ under the ordering $<$, append 
            $f$ to~$L$.
      \end{algsublist}
\item Return $L$.
\end{alglist}
\end{algorithm}

\begin{remark}
As in the similar situation in Algorithm~\ref{A:2b}, we can accomplish
step~\eqref{mediumdedupe}(b) by computing the roots of various irreducible
polynomials of degree $m$ in a fixed copy of~$\BF_{q^m}$, 
and~\cite[Theorem~1.2]{Lenstra1991} shows that we can do this in time polynomial
in $\log q$.
\end{remark}

\begin{proposition}
Algorithm~\textup{\ref{A:medium}} produces a complete set of unique 
representatives for the orbits of $\PGL_2(\BF_q)$ acting on the monic 
homogeneous degree-$n$ polynomials of the given Galois type. It runs in time 
$\Otilde(q^{n-3})$, measured in arithmetic operations in $\BF_q$.
\end{proposition}

\begin{proof}
The correctness of the algorithm is clear, because every $\PGL_2(\BF_q)$ orbit
of the given Galois type has a representative considered by the algorithm, and 
duplicates are prevented by steps~\eqref{mediumdedupe}(b)
and~\eqref{mediumdedupe}(c).

For each $m_i$ in step~\eqref{mediumdegree}, the time required to compute the
list $I_{m_i}$ is $\Otilde(q^{m_i})$, and since $m_i\le n-m_1\le n-3$ this is
$\Otilde(q^{n-3})$. The time required for step~\eqref{mediumlist} is
$\Otilde(q^{m_1-2})$, and since $m_1-2\le n-3$, this is also $\Otilde(q^{n-3})$. 
And in step~\eqref{mediumdedupe}, we consider $O(q^{n-3})$ tuples 
$(f_1,\ldots,f_r)$, and each takes time $\Otilde(1)$ to process. Thus, the total
time required is as claimed.
\end{proof}

\subsection{The case \texorpdfstring{$m_1 = n$}{m1 = n}}
\label{SS:n}
This is the first and only case in which we will require $n$ to be even. The 
case $n=4$ is covered by Theorems~\ref{T:quartics} and~\ref{T:quarticschar2}, so
we may assume that $n\ge 6$. Note that we cannot just apply
Algorithm~\ref{A:naivelist}, because that takes time $\Otilde(q^{n-2})$, and we
want an algorithm that takes time
$\Otilde(q^{n-3})$.

\begin{algorithm}
\label{A:n}
Orbit representatives for $\PGL_2(\BF_q)$ acting on the elements of
$\Sym^n(\BF_q)$ of Galois type $(n)$, where $n\ge 6$ is even.
\begin{alglist}
\algin  A prime power $q$ and an even integer $n\ge 6$.
\algout A complete set of unique representatives for the action of 
        $\PGL_2(\BF_q)$ on the monic irreducible homogenous polynomials of 
        degree~$n$.
\item \label{none}
      If $n=6$, let $M$ be the list consisting of a single monic irreducible
      cubic homogeneous polynomial in $\BF_{q^2}[x,y]$.
\item If $n=8$, let $M$ be the list consisting of the monic irreducible quartic
      homogeneous polynomials in $\BF_{q^2}[x,y]$ given by 
      Theorem~\ref{T:quartics} or Theorem~\ref{T:quarticschar2} applied to the
      field $\BF_{q^2}$.
\item \label{nthree}
      If $n\ge 10$, use Algorithm~\ref{A:naivelist} to create a list $M$ of
      orbit representatives for the action of $\PGL_2(\BF_{q^2})$ acting on the
      monic irreducible homogeneous polynomials of degree $n/2$ in 
      $\BF_{q^2}[x,y]$.
\item \label{nfour}
      Let $G$ be the list of coset representative for the left action of 
      $\PGL_2(\BF_q)$ on $\PGL_2(\BF_{q^2})$ from 
      Proposition~\ref{P:CosetRepsPGL2}.
\item \label{nfive}
      Let $N$ be the list consisting of $\Gamma(f)$, for all $\Gamma\in G$ and
      $f$ in $M$.
\item \label{nsix}
      Let $L$ be the list of all products $g g^{(q)}$ for $g\in N$, where the
      superscript $(q)$ means to raise each coefficient of a polynomial to the 
      $q$th power.
\item Let $L'$ be the list of all pairs $(\Cross f, f)$ for $f\in L$.
\item Sort $L'$, and then delete every entry $(\Cross f, f)$ where $\Cross f$
      appears as the first element of an earlier entry in $L'$.
\item \label{nnine}
      Let $L''$ be the list of second elements of the entries in $L'$.
\item Return $L''$.
\end{alglist}
\end{algorithm}

\begin{proposition}
Algorithm~\textup{\ref{A:n}} produces a complete set of unique representatives
for the orbits of $\PGL_2(\BF_q)$ acting on the monic irreducible homogeneous
polynomials of degree~$n$. It runs in time $\Otilde(q^{n-3})$, measured in 
arithmetic operations in $\BF_q$.
\end{proposition}

\begin{proof}
To prove correctness, we must show that the list $L''$ consists of unique
representatives for each $\PGL_2(\BF_q)$ orbit of irreducible polynomials of
degree $n$ over~$\BF_q$. First we show that it contains at least one 
representative from each orbit.

We know that every monic irreducible homogeneous polynomial $f$ of degree $n$ in
$\BF_q[x,y]$ can be written $g g^{(q)}$ for a monic irreducible homogeneous
polynomial $g$ in $\BF_{q^2}[x,y]$ of degree $n/2$, and $g$ is unique up to 
$g\leftrightarrow g^{(q)}$. If we can show that the list $N$ contains an element
in every orbit of $\PGL_2(\BF_q)$ acting on the left on the set $S$ of monic 
irreducible homogeneous polynomials of degree $n/2$ in $\BF_{q^2}[x,y]$, then 
that will show that $L$ contains at least one element in every orbit of
$\PGL_2(\BF_q)$ acting on the monic irreducible homogeneous polynomials of
degree $n$ in $\BF_q[x,y]$. But since $M$ is a list of representative for the 
orbits of $\PGL_2(\BF_{q^2})$ acting on~$S$, and since $G$ consists of coset
representatives for $\PGL_2(\BF_q)$ acting on $\PGL_2(\BF_{q^2})$, this is 
clear. Thus, $L$ contains a representative from each orbit of $\PGL_2(\BF_q)$
acting on the set of monic irreducible homogeneous polynomials of degree $n$
in~$\BF_q[x,y]$. In fact, $L$ contains at most two such representatives for each
orbit, because of the uniqueness of $g$ up to $g\leftrightarrow g^{(q)}$.

By construction (and by Theorem~\ref{T:invariant}), the list $L''$ contains at
most one representative from each orbit of $\PGL_2(\BF_q)$ acting on the
irreducible polynomials. But we already saw that it contains at least one such
representative. Therefore, it is a complete list of unique representatives.

The only thing left to check is that the algorithm runs in time 
$\Otilde(q^{n-3})$. Steps \eqref{none} through \eqref{nthree} take time at most
$\Otilde((q^2)^{(n/2 - 2)}) = \Otilde(q^{n-4})$. Step \eqref{nfour} takes time
$O(q^3)$, and step \eqref{nfive} takes time $\Otilde(q^{n-3})$ because there are
$O(q^3)$ elements of $G$ and $O(q^{n-6})$ elements of $M$. Steps \eqref{nsix}
through \eqref{nnine} also take time $\Otilde(q^{n-3})$, because the lists $N$,
$L$, and $L'$ contain $O(q^{n-3})$ elements. 
\end{proof}

\section{Additional efficiencies}
\label{S:speedups}

In this section we mention a few ways that the algorithms in 
Section~\ref{S:enumerating} can be improved. The asymptotic complexity of the
revised algorithms is still $\Otilde(q^{n-3})$, but the speedups in this section
improve the algorithms by constant factors.

\subsection{Avoiding repeated orbits}
Several of our algorithms include a step to deal with elements of 
$\Sym^n(\BF_q)$ that can be normalized (in the manner of the particular 
algorithm) in several ways. This happens in steps~\eqref{1dedupe}(b) and~(c) of 
Algorithm~\ref{A:1}, in steps~\eqref{2adedupe}(b) and~(c) of 
Algorithm~\ref{A:2a}, in steps~\eqref{2bdedupe}(b) and~(c) of 
Algorithm~\ref{A:2b}, and in steps~\eqref{mediumdedupe}(b) and~(c) of 
Algorithm~\ref{A:medium}. In 
Algorithm~\ref{A:medium}, for example, this is needed when there is more than
one occurrence of the number $m_1$ in the Galois type. The algorithm normalizes
elements of $\Sym^n(\BF_q)$ of the given Galois type (represented by monic
homogeneous polynomials of degree~$n$) by absorbing all of the action of 
$\PGL_2(\BF_q)$ into one factor of degree~$m_1$. If there is more than one such
factor, there is more than one normalization of the same polynomial, and the 
algorithm has to identify the resulting repeated orbits and return only one of
them. (There is also more than one normalization if a factor of degree~$m_1$ has
nontrivial $\PGL_2(\BF_q)$ stabilizer, but that is rare.)

The most straightforward way to avoid this situation is to handle Galois types
with $m_1$ occurring more than once in a different way. For instance, if we are
working with the Galois type $(4,4,3,1)$, instead of normalizing on the factors
of degree $4$ (for which there are two choices), we can instead normalize on the
factor of degree~$3$. This technique can be used to handle every Galois type 
that includes at least one value of $m$ that is at least $3$ and that occurs
just once in the type. Similarly, if the value $2$ occurs in a type exactly
twice, we can normalize the product of two irreducible quadratics.

If we have a Galois type with $m_1\ge 3$ where this is impossible --- for 
example, $(4,4,3,3)$ --- we have another option. We use a modified version of
Algorithm~\ref{A:naivelist} where we skip step~\eqref{naivelistdelete}; this
gives us a list of all monic irreducible polynomials of degree~$m_1$, grouped
by their cross polynomials. Then, in Algorithm~\ref{A:medium}, in 
step~\eqref{mediumdedupe} we do not consider \emph{all} sets of distinct 
polynomials $\{f_2,\ldots,f_r\}$; instead, we demand that the cross polynomial
of every $f_i$ whose degree is equal to $m_1$ not appear earlier on the sorted
list that that of~$f_1$. If in fact all of the additional cross polynomials are
different from the cross polynomial of~$f_1$, the orbit representative we obtain
will not be repeated unless the $\PGL_2(\BF_q)$ stabilizer of $f_1$ is
nontrivial, which is unusual (and easy to check). If some of the additional
cross polynomials \emph{are} the same as that of $f_1$, then we keep track of
this orbit on a separate list, and deduplicate this (much smaller) list
separately.

\subsection{Treating the Galois types \texorpdfstring{$(n-1,1)$}{(n-1,1)} and 
            \texorpdfstring{$(n-2,1,1)$}{(n-2,1,1)} more efficiently}
Consider the Galois type $(n-1,1)$, corresponding to a product of a linear 
homogeneous polynomial with an irreducible homogeneous polynomial of
degree~$n-1$. Instead of absorbing all the $\PGL_2(\BF_q)$ action into the 
choice of the irreducible polynomial of degree~$n-1$, we can instead demand
that the linear polynomial have its zero at $\infty$. Then we have to find
representatives for irreducible homogeneous polynomials of degree~$n-1$ up to
the $ax+b$ group. We can accomplish this by modifying the technique of 
Algorithm~\ref{A:naivelist}: We construct a copy of $\BF_{q^{n-1}}$, we choose a
basis $(\beta_1,\ldots,\beta_{n-1})$ such that $\beta_1$ appears with a nonzero
coefficient in the representation of $1$, and then we simply list the minimal
polynomials of elements whose representation on the given basis begins 
$(0,\ldots,0,1,\ldots)$, with at least one $0$ at the beginning and with the
first nonzero element being $1$. We can also take care to produce only one 
element from each Galois orbit at this stage, in order to avoid using cross
polynomials to deduplicate the list later.

For the Galois type $(n-2,1,1)$, we look for orbit representatives of the form
$xyf$ for irreducible homogeneous $f$ of degree $n-2$. We can modify $f$ only
by replacing $(x,y)$  with $(cx,y)$ or $(cy, x)$, so we can construct a copy of 
$\BF_{q^{n-2}}$ with basis $(\beta_1,\ldots,\beta_{n-2})$, list the minimal
polynomials of the elements whose representations on the given basis have their
first nonzero element equal to~$1$, and then deduplicate as usual.

\subsection{More efficient computations of \texorpdfstring{$\PGL_2$}{PGL2}
            orbits of irreducibles}
\label{SS:PGL2orbits}
In step~\eqref{mediumlist} of Algorithm~\ref{A:medium}, we obtain a list of
orbit representatives for irreducible polynomials of a given degree $m$ by using
Algorithm~\ref{A:naivelist}. If $m$ is composite, we can use a more efficient
algorithm based on the idea of Algorithm~\ref{A:n}. Namely, if $m = p m'$, we
can compute orbit representatives for $\PGL_2(\BF_{q^p})$ acting on irreducible
polynomials in $\BF_{q^p}[x]$ of degree $m'$, and then use 
Proposition~\ref{P:CosetRepsPGL2} or Proposition~\ref{P:CosetRepsPGLp} to list
coset representatives for $\PGL_2(\BF_q)$ in $\PGL_2(\BF_{q^p})$. As in 
Algorithm~\ref{A:n}, we can combine these two lists to get a complete list of
unique orbit representatives for $\PGL_2(\BF_q)$ acting on irreducible 
polynomials of degree~$n$.

In fact, what we have just shown is that if $n$ is composite, we can compute a 
complete set of unique representatives for the orbits of $\PGL_2(\BF_q)$ acting
on the monic irreducible homogeneous polynomials of degree $n$ in time 
$\Otilde(q^{n-3})$. This leaves open the case when $n$ is prime. In a followup
paper~\cite{Howe2024}, we explain a completely different technique that will
produce these orbit representatives for prime~$n$ --- indeed, for 
\emph{odd}~$n$ --- in time $\Otilde(q^{n-3})$, but no longer deterministically,
because the method relies on factoring polynomials of bounded degree in 
polynomial time.

\section{Implementations for genus \texorpdfstring{$2$}{2} and genus 
         \texorpdfstring{$3$}{3}}
\label{S:genus2}

We have implemented our algorithm for hyperelliptic curves of genus $2$ and 
genus~$3$ in Magma~\cite{magma}. Magma files with the implementations can be
found in several places: in the ancillary files attached the the arXiv version
of this paper, on the 
\href{https://ewhowe.com/papers/paper58.html}{author's web page}, and in the
GitHub repository associated to this paper~\cite{HoweRepo2024}. In addition to 
the improvements described in Section~\ref{S:speedups} and others of a similar
nature, our code for the genus-$2$ case includes an improvement for the Galois
type $(6)$ that allows us to skip the deduplication step in Algorithm~\ref{A:n}.
The basic idea is to choose the irreducible cubic polynomial in 
step~\eqref{none} of Algorithm~\ref{A:n} so that it lies in $\BF_q[x,y]$ and so 
that its zeros are permuted by the order-$3$ element $\twobytwo{0}{-1}{1}{-1}$
of $\PGL_2(\BF_q)$, so that it is easier to keep track of which elements of 
$\PGL_2(\BF_{q^2})$ in $G$ give rise to homogeneous sextic polynomials 
$f\in\BF_q[x,y]$ that lie in the same $\PGL_2(\BF_q)$ orbit. The details are too
lengthy to include here, but are spelled out in the comments in the code, as 
well as in the followup paper~\cite{Howe2024}. Other improvements are described
in the comments as well.

For our genus-$3$ implementation, we did not spend as much time optimizing, and
there are very likely improvements that can be made.

\begin{table}
\caption{Sample timings (in seconds) to compute all hyperelliptic curves of
genus~$2$ and~$3$ over $\BF_q$. The second column gives timings for Magma's 
built-in routines for genus~$2$. The third through fifth columns give timings
for the techniques of this paper: the third for computing $\PGL_2(\BF_q)$ orbit
representatives of $\Sym^6(\BF_q)$, the fourth for computing genus-$2$ curves
from these representatives, and the fifth for the total time for both. 
Similarly, the sixth through ninth columns give timings for computing genus-$3$
curves. Timings marked with an asterisk are estimates.}
\label{table}
\centering
\begin{tabular}{r r r r r r r r c c}
\toprule
&  \multicolumn{4}{c}{Genus 2}     && \multicolumn{4}{c}{Genus 3}\\
   \cmidrule(lr){2-5}                 \cmidrule(lr){7-10}
&& \multicolumn{3}{c}{This paper} &&& \multicolumn{3}{c}{This paper}\\
\cmidrule(lr){3-5}\cmidrule(lr){8-10}
$q$& 
Magma & \multicolumn{1}{c}{$\Sym^6$} & \multicolumn{1}{c}{Curves} & \multicolumn{1}{c}{Total}&&
Magma & \multicolumn{1}{c}{$\Sym^8$} & \multicolumn{1}{c}{Curves} & \multicolumn{1}{c}{Total}\\
\midrule
 $17$ &      $7.9$  &    $0.16$ &   $0.02$ &    $0.18$ &&    $5274$     & $\pz\pz  20$ & $\pz\pz  1$ & $\pz\pz  21$ \\
 $31$ &     $52.7$  &    $0.77$ &   $0.06$ &    $0.83$ &&   $99463\rs$  &    $\pz 304$ &    $\pz 14$ &    $\pz 318$ \\
 $59$ &    $327.1$  &    $3.85$ &   $0.25$ &    $4.10$ && $2408665\rs$  &       $5932$ &       $479$ &       $6411$ \\
$127$ &   $3308\z$  &   $36\zz$ &   $2\zz$ &   $38\zz$ \\
$257$ &  $27448\sz$ &  $290\zz$ &  $10\zz$ &  $300\zz$ \\
$509$ & $211655\sz $& $2307\zz$ &  $76\zz$ & $2384\zz$ \\
\bottomrule
\end{tabular}
\end{table}

We ran some timing experiments to compare our Magma code to the built-in Magma
functions that implement the algorithms of Mestre~\cite{Mestre1991} and Cardona 
and Quer~\cite{CardonaQuer2005} for genus-$2$ curves, and the algorithms of 
Lercier and Ritzenthaler~\cite{LercierRitzenthaler2012} for genus-$3$ curves. We
give some sample timings in Table~\ref{table}, taken by running Magma (V2.28-8)
on one core of an Apple M1 Max processor with 64GB RAM. For our algorithm, we 
divide our timings into two steps: computation of the $\PGL_2(\BF_q)$ orbit
representatives of $\Sym^{2g+2}(\BF_q)$ and computation of  the isomorphism
classes of curves. Our algorithm includes a computation of the automorphism
groups of the curves, which gives us a consistency check, since the sum over all
hyperelliptic curves of genus~$g$ over $\BF_q$ of $1$ over the size of the
automorphism group is equal to $q^{2g-1}$ 
\cite[Proposition~7.1]{BrockGranville2001}. 

We compare our genus-$2$ timings to those of applying the Magma command

\centerline{\texttt{Twists(HyperellipticCurveFromG2Invariants([a,b,c]))}}

\noindent
to all triples $(a,b,c)$ of elements of $\BF_q$ and then retrieving the
polynomials that define the resulting curves. For $q>127$ we estimate the time 
for the Magma builtin functions by running the above command on $10{,}000$
random triples $(a,b,c)$ and multiplying the time taken by $q^3/10^4$; these
estimates are indicated with asterisks. We see that our genus-$2$ code is
running approximately $90$ times faster than Magma's internals.

We compare our genus-$3$ timings to those of applying the Magma command

\centerline{\texttt{TwistedHyperellipticPolynomialsFromShiodaInvariants(S)}}

\noindent
to all Shioda invariants with nonzero discriminants, obtained by applying

\centerline{\texttt{ShiodaAlgebraicInvariants(V : ratsolve := true)}}

\noindent
to every element $V$ of the $5$-dimensional weighted projective space over
$\BF_q$ with weights $[2,3,4,5,6,7]$ and discarding those with discriminant~$0$. 
For $q=31$ and $q=59$ we estimate Magma's times as before. It appears that our
genus-$3$ code is running several hundred times faster than Magma's internals.

For genus-$2$ curves over $\BF_{509}$, our code spent $42.46\%$ of the time on
Galois type~$(6)$, $16.06\%$ on type $(3,3)$, and $12.55\%$  on type $(5,1)$,
with the remaining $29\%$ of the time divided among the remaining eight types.
For genus-$3$ curves over $\BF_{59}$, our code spent $26.49\%$ of the time on 
Galois type $(7,1)$, $25.77\%$ on type $(8)$, $19.45\%$ on type $(2,2,2,2)$, and
$7.87\%$ on type $(4,4)$, with the remaining $20\%$ of the time divided among 
the remaining eighteen types.

We note that memory handling issues may have slowed the genus-$3$ computation
for $q=59$, which reemphasizes the point, made in the introduction, that it
would be good to have a version of our algorithm for higher genera that requires
less space. We have not yet implemented the low-memory algorithm 
from~\cite{Howe2024} to see whether that will help improve our timings for 
larger~$q$.


\bibliographystyle{hplaindoi}
\bibliography{enumerating}

\end{document}